\documentclass[11pt, reqno]{amsart}
\usepackage[left=1in, right=1in]{geometry}
\usepackage{amsmath, amsthm, amssymb}
\usepackage{graphicx}
\usepackage[colorlinks=true]{hyperref}
\usepackage{xcolor}
\usepackage{mathtools}
\usepackage[english]{babel}
\usepackage{booktabs}
\usepackage{enumitem}
\usepackage{threeparttable}

\theoremstyle{plain}
\newtheorem{theorem}{Theorem}[section]
\newtheorem{lemma}{Lemma}[section]
\newtheorem{remark}{Remark}[section]

\newtheorem{proposition}{Proposition}[section]

\numberwithin{equation}{section}

\allowdisplaybreaks[4]

\begin{document}
	\title[Chern-Simons gauged $O(3)$ sigma equations]{Almost optimal well-posedness for Chern--Simons gauged $O(3)$ sigma model under the Lorenz gauge}
	
	\author{HUALI ZHANG AND JIE ZHOU}
	
	\date{}
	
	\begin{abstract}	

	In this paper, we study the low-regularity Cauchy problem for the Chern--Simons gauged $O(3)$ sigma model in $\mathbb{R}^{1+d}$ ($d=1,2$) under the Lorenz gauge.

	For $d=1$, we establish local well-posedness for initial data $(\boldsymbol{\phi}_0,\mathbf{A}_0)\in H^{s_1}(\mathbb{R})\times H^{s_1-1}(\mathbb{R})$ with $s_1>\frac12$. This improves the previous result of Jin and Huh \cite{HJ} by one quarter of a derivative and is almost optimal in view of the scaling-invariant regularities $\dot H^{1/2}(\mathbb{R})$ for the matter field and $\dot H^{-1/2}(\mathbb{R})$ for the gauge field.

	For $d=2$, we establish local well-posedness for initial data $(\boldsymbol{\phi}_0,\mathbf{A}_0)\in H^{s_2}(\mathbb{R}^2)\times H^{s_2-\frac34}(\mathbb{R}^2)$ with $s_2>1$. This improves the previous result of Jin and Zhang \cite{JZ} by one quarter of a derivative and brings the regularity threshold close to the scaling-invariant exponents $\dot H^{1}(\mathbb{R}^2)$ and $\dot H^{0}(\mathbb{R}^2)$ for the matter and gauge fields, respectively.

	The analysis relies on two main ingredients. In two space dimensions, we identify the complete null structure of the derivative nonlinearities, allowing the entire system to be treated within a unified null-form framework. In one space dimension, we establish a direct energy estimate in the function space introduced by Keel and Tao, avoiding the finite-propagation reduction to a small-data problem and enabling the low-regularity iteration for general initial data.

	\end{abstract}
	\maketitle
	
	\section{Introduction}
	
	In this paper, we investigate the low-regularity Cauchy problem for the Chern--Simons gauged $O(3)$ sigma system under the Lorenz gauge in one and two spatial dimensions. The model originates from the classical $O(3)$ sigma model in quantum field theory and shares close analogies with the Yang--Mills and Yang--Mills--Higgs equations. The pure (1+2)-dimensional $O(3)$ sigma model has been extensively studied in both mathematics \cite{RS} and theoretical physics \cite{L,ATY, GG, KLL} because of its rich geometric structure and scale invariance. However, the scale invariance allows solitons to change their size during time evolution without any energy cost, preventing them from serving as realistic particle models. To overcome this drawback, one introduces gauge field dynamics through either the Maxwell or the Chern--Simons action, thereby stabilizing the soliton size. In this work, we focus on the latter case, commonly refer red to as the Chern--Simons gauged $O(3)$ sigma system.
	
	More precisely, the Euler--Lagrange equations of the Chern--Simons gauged \(O(3)\) sigma system in \(\mathbb{R}^{1+2}\) take the form
	\begin{align}
		D_{\mu} D^{\mu} \boldsymbol{\phi}+\boldsymbol{\phi}\left(D_{\mu} \boldsymbol{\phi} \cdot D^{\mu} \boldsymbol{\phi}\right) 
		& = - \frac{1}{\kappa^{2}}\left(\boldsymbol{\phi}\left(n \cdot \boldsymbol{\phi}\right)-n(\boldsymbol{\phi} \cdot \boldsymbol{\phi})\right)\left(1-n \cdot \boldsymbol{\phi}\right)^{2}\left(1+2 n \cdot \boldsymbol{\phi}\right),  
		\label{2EL1}
		\\
		\kappa F_{\nu \rho} 
		& = - \epsilon_{\mu \nu \rho}(\boldsymbol{\phi} \times D^{\mu} \boldsymbol{\phi}) \cdot n,
		\label{2EL2}
	\end{align}
	subject to the initial data
	\begin{equation}\label{2ini}
		\boldsymbol{\phi}(0, \cdot)=\boldsymbol{\phi}_{0}, 
		\quad \partial_{t}\boldsymbol{\phi}(0,\cdot)=\boldsymbol{\phi}_1, 
		\quad A_{\mu}(0, \cdot)=a_{\mu}.
	\end{equation}
	Here $D_\mu\boldsymbol{\phi}=\partial_\mu\boldsymbol{\phi}+A_\mu(n\times\boldsymbol{\phi}), F_{\mu\nu}=\partial_\mu A_\nu-\partial_\nu A_\mu$ denote the covariant derivative and the electromagnetic field tensor, respectively. The matter field $\boldsymbol{\phi}=(\phi_1,\phi_2,\phi_3)^T$ takes values in the unit sphere $\mathbb S^2$, namely $\phi_1^2 + \phi_2^2 + \phi_3^2 = 1$, while $A_\mu$ is a real-valued gauge field with $\mathbf{A}=(A_0,A_1,A_2)^T$. Throughout the paper, we denote $n=(0,0,1)^T$. $\kappa>0$ is the Chern--Simons coupling constant, and $\epsilon^{\mu\nu\rho}$ is the totally antisymmetric tensor with  $\epsilon^{012}=1$.

	\subsection{Historical results}
	
	Rigorous mathematical studies of the Chern--Simons gauged $O(3)$ sigma model have primarily focused on its static regime. For both the symmetric and asymmetric vacuum cases, the governing equations reduce to a elliptic system, and extensive results have been obtained on the existence, uniqueness, classification, and qualitative properties of soliton  solutions. 
	
	For the symmetric vacuum, Yang \cite{Yang} established the existence of topological solutions and radially symmetric nontopological solutions, while Choe and Nam \cite{CN} proved the uniqueness of topological solutions for sufficiently small or large Chern--Simons coupling constants $\kappa$. Choe and Han \cite{CH} subsequently obtained a complete classification of radially symmetric solutions together with precise asymptotic behavior.
	
	In the asymmetric vacuum, Choe et al. \cite{CHLL1} classified all finite-energy solutions into topological, type-I nontopological and type-II nontopological solutions, and established their uniqueness and structural properties. Building on this classification, subsequent works constructed bubbling solutions \cite{CHLL2} and analyzed the asymptotic behavior of general solutions \cite{CHLL1}. More recently, Chern, Chen and Shen \cite{CCS} resolved the remaining open problem concerning type-II nontopological solutions, thereby completing the classification theory.

	Despite the substantial progress on the static theory, much less is known about the corresponding dynamical problem, namely the Cauchy problem for the CS-$O(3)$ sigma system. In particular, low-regularity well-posedness, especially near the scaling-critical regularity, remains largely open. To adress this issue, one must first impose a gauge condition, as the system is invariant under gauge transformations: 
	\begin{equation} \label{gaugetrans}
		\boldsymbol{\phi} = (z, \phi_3) \rightarrow(z e^{\operatorname{i} \chi}, \phi_3), \quad A_{\mu} \rightarrow A_{\mu}-\partial_{\mu} \chi,
	\end{equation}
	where $\chi$ is a real valued smooth function on $\mathbb{R}^{1+d}$ and  $z=\phi_1+\operatorname{i} \phi_2$. Consequently, a solution to \eqref{2EL1}--\eqref{2EL2} is formed by a class of gauge-equivalent pairs $(\boldsymbol{\phi}, A_{\mu})$. 
	
	We work exclusively under the Lorenz gauge $\partial_{\mu} A^{\mu}=0$, which preserves the Lorentz covariance of the equations and allows the CS-$O(3)$ sigma system to be reformulated as a coupled system of nonlinear wave equations. Unlike the Coulomb gauge, where part of the dynamics is governed by elliptic equations, the Lorenz gauge leads to a purely hyperbolic formulation. The absence of elliptic regularization makes the Lorenz-gauge problem substantially more challenging.
	
	Exploiting this algebraic property, we rigorously reformulate the system \eqref{2EL1}--\eqref{2EL2} as:
	\begin{align}
		\begin{split}\label{cso3s1}
			\square \boldsymbol{\phi} = 
			& -\boldsymbol{\phi}(\partial_\mu \phi \cdot \partial^\mu \phi+2 A_{\mu} \partial^{\mu} \boldsymbol{\phi} \cdot(n \times \boldsymbol{\phi})+A_{\mu} A^{\mu}|n \times \boldsymbol{\phi}|^{2})-2 A^{\mu}(n \times \partial_{\mu} \boldsymbol{\phi}) 
			\\ 
			& -A_{\mu} A^{\mu} n \times(n \times \boldsymbol{\phi}) - V(\boldsymbol{\phi}), 
		\end{split}
			\\
		\begin{split}\label{cso3s2}
			\square A_{\mu} =
			& ~ \frac{1}{\kappa} \epsilon_{\mu \nu \rho}(n \cdot(\partial^{\nu} \boldsymbol{\phi} \times \partial^{\rho} \boldsymbol{\phi})+A^{\rho} \partial^{\nu}((\boldsymbol{\phi} \times(n \times \boldsymbol{\phi})) \cdot n)),
		\end{split}
	\end{align}
	with
	\begin{equation}\label{2ini2}
		\boldsymbol{\phi}(0, \cdot)=\boldsymbol{\phi}_{0}, 
		\quad \partial_{t}\boldsymbol{\phi}(0,\cdot)=\boldsymbol{\phi}_1, 
		\quad A_{\mu}(0, \cdot)=a_{\mu}, 
		\quad \partial_{t} A_{\mu}(0, \cdot)=\dot{a}_{\mu}.
	\end{equation}
	The potential is given by $V(\boldsymbol{\phi}) = \frac{1}{\kappa^{2}}\left(\boldsymbol{\phi}\left(n \cdot \boldsymbol{\phi}\right)-n(\boldsymbol{\phi} \cdot \boldsymbol{\phi})\right)\left(1-n \cdot \boldsymbol{\phi}\right)^{2}\left(1+2 n \cdot \boldsymbol{\phi}\right) $.
	 
	The same procedure also applies to the one-dimensional reduction. Although the reduced system no longer retains the original physical interpretation,  it preserves the essential gauge structure and the null structure of the derivative terms. Consequently, it serves as a natural model for investigating low-regularity well-posedness. The corresponding wave system is given by
	\begin{align}
		\begin{split}
			\Box \boldsymbol{\phi} 
			& = - \boldsymbol{\phi} \left( \partial_{\mu} \boldsymbol{\phi} \cdot \partial^{\mu} \boldsymbol{\phi} + 2 A_{\mu} \partial^{\mu} \boldsymbol{\phi} \cdot (n \times \boldsymbol{\phi}) + A_{\mu} A^{\mu} |n \times \boldsymbol{\phi}|^2 + \phi_3 U(\phi_3, N)  \right)
			\\
			& - 2 A_{\mu} \partial^{\mu}(n \times \boldsymbol{\phi}) - A_{\mu} A^{\mu} n \times (n \times \boldsymbol{\phi}),
			\label{1cso3s1}
		\end{split}
		\\
		\kappa \Box A_0
		& = \partial_1(N|n \times \boldsymbol{\phi}|^2),
		\label{1cso3s2}
		\\
		\kappa \Box A_1
		& = \partial_0(N|n \times \boldsymbol{\phi}|^2),
		\label{1cso3s3}
		\\
		\kappa \Box N
		& = D_1(n \times \boldsymbol{\phi}) \cdot D_0 \boldsymbol{\phi} - D_0(n \times \boldsymbol{\phi}) \cdot D_1 \boldsymbol{\phi} - (\phi_1^2 + \phi_2^2)F_{01},
		\label{1cso3s4}
	\end{align}
	with the initial data
	\begin{equation}\label{1ini2}
		\begin{aligned}
			& \boldsymbol{\phi}(0, \cdot) = \boldsymbol{\phi}_0, \quad \partial_t \boldsymbol{\phi}(0, \cdot) = \boldsymbol{\phi}_1, \quad A_{\mu}(0, \cdot) = a_{\mu}, \quad \partial_t A_{\mu}(0, \cdot) = \dot{a}_{\mu}, 
			\\
			& N(0, \cdot) = n_0, \quad \partial_t N(0, \cdot) = - \langle n \times \boldsymbol{\phi}_0, \partial_1 \boldsymbol{\phi}_0 + a_1(n \times \boldsymbol{\phi}_0) \rangle.
		\end{aligned}
	\end{equation}
	where $N$ is a real field. We refer the reader to the appendix (see section \ref{appendix}) for the detailed derivation of \eqref{cso3s1}--\eqref{cso3s2} and \eqref{1cso3s1}--\eqref{1cso3s2}. This reformulation places the problem within the general framework of semilinear wave equations with derivative nonlinearities. We refer the reader to \cite{AnKapi,DFW,FW,GN,KM1,KM2,Ta1,Ta2,Zhang,Zhou2} for related developments.
	
	Regarding the well-posedness of the above systems, Huh and Jin \cite{HJ} studied local solutions for $(\boldsymbol{\phi}_0, a_{\mu}) \in H^{\frac{3}{4}+}(\mathbb{R}) \times H^{\frac{1}{4}+}(\mathbb{R})$ and further proved global well-posedness for $(\boldsymbol{\phi}_0, a_{\mu}) \in H^1(\mathbb{R}) \times H^{\frac{1}{2}+}(\mathbb{R})$ by energy conservation when $d=1$. In the case of $d=2$, the local well-posedness was first established by Jin and Zhang \cite{JZ} with initial data $(\boldsymbol{\phi}_0, a_{\mu}) \in H^{1+}(\mathbb{R}^2) \times H^{\frac{1}{2}+}(\mathbb{R}^2)$.
	
	Having introduced the Chern--Simons gauged $O(3)$ sigma system, we next place it in the broader context of Chern--Simons gauge theories by comparing it with several closely related models, including the Chern--Simons--Higgs (CSH), Chern--Simons--Dirac (CSD), and Chern--Simons gauged nonlinear Schr\"{o}dinger (CSS) models. Although these systems share the same gauge structure, they dy in iffer substantiallthe differential order of the matter equations. More precisely, the CS-$O(3)$ sigma and CSH models both involve second-order derivatives, the CSD model features first-order derivatives, while the CSS system combines a first-order time derivative with second-order spatial derivatives. These structural differences lead to distinct analytical difficulties in the corresponding low-regularity well-posedness. 
	
	The Cauchy problem for the CSH system has been extensively studied very recently. Under the Lorenz gauge, Huh \cite{Huh1} first established local well-posedness for initial data $ (\boldsymbol{\phi}_0, \mathbf{A}_0) \in H^{\frac{5}{4}+}(\mathbb{R}^2) \times H^{\frac{3}{4}+}(\mathbb{R}^2) $. This result was subsequently improved by Selberg and Tesfahun \cite{ST}, and finally by Huh and Oh \cite{HO}, who introduced abstract bilinear null forms to obtain local well-posedness for $(\boldsymbol{\phi}_0, \mathbf{A}_0) \in H^{\frac34+}(\mathbb{R}^2) \times H^{\frac14+}(\mathbb{R}^2)$. More recently, Huh \cite{Huh2} proved local well-posedness under the Coulomb gauge with  $(\boldsymbol{\phi}_0, \mathbf{A}_0) \in H^{1+}(\mathbb{R}^2) \times H^{0+}(\mathbb{R}^2)$, which is almost optimal with respect to the regularity of the gauge field.
	
	Similar progress has been made for the CSD system. Huh \cite{Huh4} obtained the local well-posedness result under the Lorenz gauge, which was later improved by Huh and Oh \cite{HO} through the use of abstract bilinear null forms. The best known result is $ (\boldsymbol{\phi}_0, \mathbf{A}_0) \in {H}^{\frac{1}{4}+}(\mathbb{R}^2) \times {H}^{\frac{1}{4}+}(\mathbb{R}^2)$. Under the Coulomb gauge, Huh \cite{Huh4} established local well-posedness for $(\boldsymbol{\phi}_0, \mathbf{A}_0 ) \in H^{\frac{1}{2}+}(\mathbb{R}^2) \times L^2(\mathbb{R}^2)$, where the regularity of the gauge field is scaling optimal.
	
	For the planar CSS system, the gauge potentials are uniquely determined by the matter field and therefore the Cauchy problem we only focus only on the regularity of $\boldsymbol{\phi}$. Under the Coulomb gauge, Berg\'{e}, Bouard and Saut\cite{BBS} proved local well-posedness when $\boldsymbol{\phi}_0 \in H^2(\mathbb{R}^2)$, which was subsequently improved by Huh \cite{Huh5} to the energy space $H^1(\mathbb{R}^2)$ by using Strichartz estimates. Lim \cite{Lim} further lowered the regularity to $\boldsymbol{\phi}_0 \in H^{1+}(\mathbb{R}^2)$, while Liu, Smith and Tataru \cite{LST} employed the $U^2$ and $V^2$ spaces to prove local well-posedness for small initial data in $H^{0+}(\mathbb{R}^2)$ under the heat gauge.
	
	The above comparison shows that low-regularity well-posedness is by now well understood for several fundamental Chern--Simons gauge theories, while the corresponding theory for the CS-$O(3)$ sigma system remains far from complete. This gap provides the main motivation for the present work. For convenience, the current best results together with the corresponding scaling-invariant regularities are summarized in Table~\ref{1}.
	\begin{table}[htbp]
		\centering
		\caption{Current results and scaling-invariant spaces for $n=2$}
		\label{1}
		\begin{tabular}{l c c c c}
			\toprule
			\textbf{Systems} & Lorenz gauge & Coulomb gauge & Heat gauge & scaling-invariant spaces
			\\
			\midrule
			CSH  & $ H^{\frac{3}{4}+} \times H^{\frac{1}{4}+}$ & $ H^{1+} \times H^{0+}$ & -- &$\dot{H}^{\frac{1}{2}} \times L^2$
			\\
			CSD  & $H^{\frac{1}{4}+} \times H^{\frac{1}{4}+}$  & $H^{\frac{1}{2}+} \times L^2$ & -- &$L^2 \times L^2$
			\\
			CSS  & --      & $H^{1+}$     & $H^{0+}$ \small(small data) &$L^2$
			\\
			\bottomrule
		\end{tabular}
	\end{table}

	\subsection{Motivation}
	

	In the decoupling limit $\mathbf{A}=0$, the systems \eqref{cso3s1}--\eqref{cso3s2} and \eqref{1cso3s1}--\eqref{1cso3s2} reduce to the classical wave maps equation. The low-regularity Cauchy problem for wave maps has been extensively studied, and local well-posedness is now available at regularities arbitrarily close to the scaling-invariant threshold $s_c=d/2$. In three space dimensions, Klainerman and Machedon~\cite{KM2} initiated the low-regularity analysis of wave maps by uncovering the underlying null structure. This fundamental observation led to bilinear estimates without loss of derivatives and established local well-posedness for initial data in $H^{3/2+}\times H^{1/2+}$. In one space dimension, Keel and Tao~\cite{KT} introduced the null-coordinate framework and proved almost-critical local well-posedness for initial data in $H^{\frac12+}\times H^{-\frac12+}$. For other dimensional results, we refer the reader to the works of Tataru, Tataru-Sterbenz, and Zhou \cite{STa,Ta3,Ta4,Zhou2}. More recently, Zhou~\cite{WZ,Zhou5} developed a new div--curl approach for wave maps, providing an alternative physical-space method for capturing the essential bilinear cancellations.

	For the Chern--Simons gauged $O(3)$ sigma model, the scaling-invariant regularities for the matter and gauge fields are $\dot{H}^{d/2}(\mathbb{R}^d)$ and $\dot{H}^{d/2-1}(\mathbb{R}^d)$, respectively. Existing well-posedness results for the coupled system, such as \cite{HJ} in one space dimension and \cite{JZ} in two space dimensions, require substantially higher regularity than these scaling-invariant thresholds. This discrepancy naturally raises a fundamental question: does the gauge coupling genuinely introduce an intrinsic regularity barrier, or does the apparent loss of regularity merely reflect limitations of the analytical techniques employed so far? Addressing this question is the primary motivation of the present work.

	A closer examination of the previous analyses suggests that the main obstacle lies not in the strength of the nonlinear coupling itself, but in the incomplete exploitation of the algebraic structure of the system. In two space dimensions, Jin and Zhang \cite{JZ} identified the null structure hidden in the interaction $A_{\mu}\partial^{\mu}\boldsymbol{\phi}$, while the remaining derivative nonlinearities $\epsilon_{\mu\nu\rho}A^{\rho}\partial^{\nu}\boldsymbol{\phi}$ and $A_{\mu}A^{\mu}$ were treated without taking advantage of any additional cancellation. Consequently, a significant portion of the null structure inherent in the coupled system remained unexplored, leaving open the possibility that the existing regularity assumptions are not optimal.
	
	The one-dimensional case presents a different difficulty. Even though the full null structure is available, the product estimates in wave--Sobolev spaces are insufficient to exploit these cancellations near the scaling-critical regularity. Consequently, the nonlinear iteration cannot be closed within this framework at the desired regularity, which explains the limitation of the result obtained by Jin and Huh \cite{HJ}. This indicates that a different functional framework, better adapted to the null geometry of the equations, is required.
	
	These observations indicate that further progress depends on combining a more complete understanding of the hidden null structures with analytical tools better suited to exploiting them. Guided by this perspective, the present paper develops a new formulation of the Chern--Simons gauged $O(3)$ sigma model that reveals the full null structure of the derivative nonlinearities, together with an analytical framework capable of taking advantage of these cancellations. This ultimately leads to improved local well-posedness results in both one and two space dimensions.

	\subsection{Statement of the main results}
	
	Let us state our main results concerning the local well-posedness of the CS-$O(3)$ sigma system in one and two spatial dimensions. 
	
	\begin{theorem}
		\label{main thm 1dim}
		For $d=1$ and $s>\frac12$, suppose the initial data in the following Sobolev spaces:
		\begin{gather*}
			\boldsymbol{\phi}(0, \cdot) = \boldsymbol{\phi}_0 \in H^{s}(\mathbb{R}), \quad \partial_t \boldsymbol{\phi}(0, \cdot) = \boldsymbol{\phi}_1 \in H^{s-1}(\mathbb{R}),
			\\
			A_{\mu}(0, \cdot)=a_{\mu} \in H^{s-1}(\mathbb{R}), \quad N(0, \cdot) = n_0 \in H^{s}(\mathbb{R}), 
		\end{gather*}
		with $\langle \boldsymbol{\phi}_0, \boldsymbol{\phi}_1 \rangle = 0$. Then the initial value problem \eqref{1cso3s1}--\eqref{1ini2} is locally well-posed. More precisely, there exist a time $T>0$ depending continuously on the data such that the solution satisfies
		\begin{gather*}
			\boldsymbol{\phi} \in C([0, T] ; H^{s}(\mathbb{R}))  \cap C^1( [0, T] ; H^{s-1}(\mathbb{R})),
			\\
			A_{\mu} \in C([0, T] ; H^{s-1}(\mathbb{R})),
			\\
			N \in C([0, T] ; H^{s}(\mathbb{R})),
		\end{gather*}
	\end{theorem}
	
	\begin{theorem}
		\label{main thm}
		For $d=2$ and $s>1$, suppose the initial data in the following Sobolev spaces:
		\[
		\boldsymbol{\phi}(0, \cdot)=\boldsymbol{\phi}_{0} \in H^{s}(\mathbb{R}^{2}), \quad \partial_{t} \boldsymbol{\phi}(0, \cdot)=\boldsymbol{\phi}_{1} \in H^{s-1}(\mathbb{R}^{2}), \quad  A_{\mu}(0, \cdot)=a_{\mu} \in H^{s-\frac{3}{4}}(\mathbb{R}^{2}), 
		\]
		with $\langle \boldsymbol{\phi}_0, \boldsymbol{\phi}_1 \rangle = 0$. Then the initial value problem \eqref{2ini}--\eqref{cso3s2} is locally well-posed. More precisely, there exist a time $T>0$ depending continuously on the data such that the solution satisfies
		\begin{gather*}
			\boldsymbol{\phi} \in C([0, T] ; H^{s}(\mathbb{R}^{2}))  \cap C^1([0, T] ; H^{s-1}(\mathbb{R}^{2})),
			\\
			A_{\mu} \in C([0, T] ; H^{s-\frac{3}{4}}(\mathbb{R}^{2})).
		\end{gather*}
	\end{theorem}
	
	\begin{remark}
		System \eqref{1cso3s1}--\eqref{1cso3s4} and \eqref{cso3s1}--\eqref{cso3s2} is invariant under the scaling
		\begin{equation*}
			\phi(t, x) \rightarrow \phi^\lambda:=\phi(\lambda t, \lambda x), \quad 
			A_\mu \rightarrow A_\mu^\lambda:=\lambda A_\mu(\lambda t, \lambda x), \quad 
			N \rightarrow N^\lambda(t, x):= \lambda N(\lambda t, \lambda x).
		\end{equation*}
		This gives
		\[
		\begin{aligned} 
			&\|\boldsymbol{\phi}^\lambda\|_{\dot{H}^{s_c}(\mathbb{R}^d)}=\lambda^{s_c-\frac{d}{2}}\|\boldsymbol{\phi}\|_{\dot{H}^{s_c}(\mathbb{R}^d)},
			\\ 
			&\|A_\mu^\lambda\|_{\dot{H}^{s_c}(\mathbb{R}^d)}=\lambda^{s_c+1-\frac{d}{2}}\|A_\mu\|_{\dot{H}^{s_c}(\mathbb{R}^d)},
			\\
			&\|N^\lambda\|_{\dot{H}^{s_c}(\mathbb{R}^d)}=\lambda^{s_c+1-\frac{d}{2}}\|N\|_{\dot{H}^{s_c}(\mathbb{R}^d)}.
		\end{aligned}
		\]
		Hence the scaling-inavriant Sobolev regularities are $s_c=\frac{d}{2}$ for $\boldsymbol{\phi}$ and $s_c=\frac{d}{2}-1$ for $A_{\mu}$ and $N$. Then it is expected that local well-posedness holds for $(\boldsymbol{\phi}_0, a_{\mu}, n_0) \in H^{\frac12+}(\mathbb{R}) \times H^{-\frac12+}(\mathbb{R}) \times H^{\frac12+}(\mathbb{R})$ and $(\boldsymbol{\phi}_0, a_{\mu}) \in H^{1+}(\mathbb{R}^2) \times H^{0+}(\mathbb{R}^2)$. Theorem~\ref{main thm 1dim} yields almost scaling-invariant local well-posedness for both $\boldsymbol{\phi}$ and $A_\mu$, while Theorem~\ref{main thm} yields almost scaling critical local well-posedness for $\boldsymbol{\phi}$. 
	\end{remark}
	
	\begin{remark}
		The Sobolev space $H^s$ is not directly applicable to functions taking values in the sphere $\mathbb{S}^2$, since they do not decay at spatial infinity, and consequently such functions do not belong to $L^{2}$. Therefore, when we write  $\boldsymbol{\phi}(t, \cdot) \in H^{s}$, it is understood that  $\boldsymbol{\phi}(t, \cdot)-n \in H^{s}$.
	\end{remark}
	
	\subsection{Main ideas of the proof}
	
	The proofs of Theorems~\ref{main thm 1dim} and~\ref{main thm} follow the same general strategy. Under the Lorenz gauge, we consider the corresponding nonlinear wave systems and establish the linear and nonlinear estimates in suitable function spaces. These estimates enable us to construct the associated solution map and prove that it is a contraction on a suitable complete metric space. The desired local well-posedness then follows from the Banach fixed-point theorem.
	
	More precisely, for the two-dimensional problem, the main difficulty lies in the derivative nonlinearities of the wave system. Our key contribution is to uncover the hidden null structures in the remaining nonlinear interactions. Motivated by the approach of Huh~\cite{Huh2}, we introduce suitable auxiliary vector fields $B_{\mu}$, which allow us to rewrite every derivative nonlinearity in the wave formulation as a null form. This complete null-form decomposition reveals the intrinsic cancellations of the system and reduces the nonlinear analysis to bilinear estimates for null forms. Combined with Selberg's linear estimate in $\mathcal{H}^{s,b}$~\cite{Selberg}, for the linear wave equation
	\begin{equation*}
		\begin{cases}
			\square w = F(t,x), \quad (t,x) \in \mathbb{R}^{1+d}, \\
			w|_{t=0} = f, \quad \partial_t w|_{t=0} = g,
		\end{cases}
	\end{equation*}
	we have
	\begin{equation*}
		|w|_{s, b} \leq C_{0}\left( \|f\|_{H^{s}}+\|g\|_{H^{s-1}}+T^{\frac{\epsilon}{4}}\|F\|_{s-1, b+\epsilon-1}\right) ,
	\end{equation*}
	Together with the bilinear estimates in the wave--Sobolev spaces $H^{s, b}$ for null forms, this enables us to prove well-posedness under relaxed regularity assumptions on the initial data.

	For the one-dimensional problem, the null-form wave formulation is already available, making the main task analytical rather than structural. Since standard product estimates in wave--Sobolev spaces fall short of reaching the critical regularity threshold $s>\frac12$, we adopt the null-coordinate framework of Keel and Tao~\cite{KT} and work in the anisotropic spaces $X^{s,b}$. Crucially, we demonstrate that favorable bilinear estimates hold not only for the $Q_0$-type null forms, but also for the $Q_{10}$-type (see \eqref{Q}, \eqref{Q0}, \eqref{Q1} below). Unlike in~\cite{KT}, however, we do not reduce the problem to the small-data regime via finite propagation speed. Instead, we establish a new linear energy estimate directly in $X^{s,b}$,
	\begin{equation*}
		\|w\|_{X^{s,b}} \leq C_0 \left( \|f\|_{H^s} + \|g\|_{H^{s-1}} + T^{\epsilon/4} \|F\|_{X^{s-1, b+\epsilon-1}} \right),
	\end{equation*}
	which, together with the sharp product estimates for null forms in $X^{s,b}$~\cite{KT}, provides the bilinear estimates needed for the fixed-point argument and yields local well-posedness under relaxed regularity assumptions.

	\subsection{Notations}
	Greek indices, such as $\mu, \nu, \rho $ range from 0 to $d$, while Roman indices such as $ i,j,k $ range from 1 to $d$. Spacetime derivatives are denoted by $ \partial_\alpha = \partial_{x^{\alpha}} $, with $ (t, x) = (x^{\alpha})_{0\leq \alpha \leq d} $. Indices are raised and lowered using the Minkowski metric $g=(g_{\mu\nu})_{(1+d) \times (1+d)}=diag(1, -1, \cdots,  -1)$. We adopt the Einstein summation convention, where we sum over repeated upper and lower indices. Therefore, $\square = \partial_{t}^{2}-\Delta = \partial^{\alpha} \partial_{\alpha}$. 
	
	Throughout the paper, $s+$ denotes $s+\epsilon$ for arbitrarily small  positive constant $\epsilon$. The notations $a \lesssim b$ means $a \leq Cb$ for some positive constant $C$.
	
	We set cut-off functions $\theta$ and $\eta$ satisfying
	\begin{equation}\label{theta}
		\theta \in C_{c}^{\infty}(\mathbb{R}), \quad \theta=1 \text { on }[-1,1], \quad \operatorname{supp} \theta \subseteq(-2,2)
	\end{equation}
	and
	\begin{equation}\label{eta}
		\eta \in C_{c}^{\infty}(\mathbb{R}), \quad \eta=1 \text { on }[-2,2], \quad \operatorname{supp} \eta \subseteq(-4,4).
	\end{equation}
	We write $\theta_T (t) \coloneqq \theta(t/T)$.
	
	For $f \in \mathcal{S}^{\prime}(\mathbb{R}^d)$ and $g \in \mathcal{S}^{\prime}(\mathbb{R}^{1+d}) $, the Fourier transforms $\widehat{\cdot }$ and $\widetilde{\cdot}$ are defined by
	\[
	\widehat{f}(\xi)=\int_{\mathbb{R}^d} \mathrm{e}^{\mathrm{i} x \cdot \xi} f(x) d x, \quad 
	\widetilde{g}(\tau, \xi)=\int_{\mathbb{R}} \int_{\mathbb{R}^d} \mathrm{e}^{\mathrm{i}(t \tau+x \cdot \xi)} g(t, x) d x d t,
	\]
	where $\mathcal{S}^{\prime}$ denotes the space of tempered distributions, the dual of Schwartz space $\mathcal{S}$. For any $g \in \mathcal{S}^{\prime}(\mathbb{R}^{1+d}) $, it makes sense to restrict $g$ to any slice $\{t\} \times \mathbb{R}^d$.  We write $u \precsim v$ denotes $ |\widehat{u}| \lesssim \widehat{v} $.
	
	Let $\Lambda^{\alpha}, \Lambda_{+}^{\alpha}, \Lambda_{-}^{\alpha}$ be the Fourier multipliers given by
	\[
	\widehat{\Lambda^\alpha f}(\xi) = \langle \xi \rangle^{\alpha} \widehat{f}(\xi), \quad
	\widetilde{\Lambda_+^\alpha g}(\tau, \xi) = \langle |\tau| + |\xi| \rangle^{\alpha}\widetilde{g}(\tau, \xi), \quad
	\widetilde{\Lambda_-^\alpha g}(\tau, \xi) = \langle |\tau| - |\xi| \rangle ^{\alpha} \widetilde{g}(\tau, \xi),
	\]
	where $\langle \cdot \rangle = 1 + | \cdot |$. The homogeneous versions of these operators are $D^{\alpha}, D_{+}^{\alpha}, D_{-}^{\alpha}$ with symbols $ |\xi|^{\alpha}, (|\tau|+|\xi|)^{\alpha}, ||\tau|-|\xi||^{\alpha} $ respectively.
	
	For the spatial domain $\mathbb{R}^d$, the standard Sobolev space $H^s(\mathbb{R}^d)$ is defined by 
	$$ 
	H^s(\mathbb{R}^d) = \{ f \in \mathcal{S}'(\mathbb{R}^d) : \|f\|_{H^s} := \|\Lambda^s f\|_{L^2(\mathbb{R}^d)} < \infty \}. 
	$$
	
	For functions on the spacetime domain $\mathbb{R}^{1+d}$, we introduce the wave--Sobolev spaces $H^{s,b}$ and $\mathcal{H}^{s,b}$ equipped with the norms
	\[
	\|\phi\|_{s,b} = \|\Lambda^s \Lambda_-^b \phi\|_{L^2(\mathbb{R}^{1+d})}, 
	\]
	and
	\[
	|\phi|_{s,b} = \|\phi\|_{s,b} + \|\partial_t \phi\|_{s-1,b} \sim \|\Lambda^{s-1} \Lambda_+ \Lambda_-^b \phi\|_{L^2(\mathbb{R}^{1+d})}. 
	\]
	According to Selberg \cite{Selberg}, for $b > \frac{1}{2}$, we have the following continuous embeddings:
	\begin{gather*}
		H^{s,b} 
	 	\hookrightarrow C(\mathbb{R}; H^s(\mathbb{R}^d)),
		\\
		\mathcal{H}^{s,b} 
		\hookrightarrow C(\mathbb{R}; H^s(\mathbb{R}^d)) \cap C^1(\mathbb{R}; H^{s-1}(\mathbb{R}^d)). 
	\end{gather*}
	
	The time-resticted space is defined as
	\[
	\mathcal{H}^{s,b}_T = \{ \phi = \Phi|_{[0,T]\times\mathbb{R}^d} : \Phi \in  \mathcal{H}^{s,b}\},
	\]
	with norm
	\[
	|\phi|_{s, b, T} = \inf \left\{|\Phi|_{s,b} :  \Phi|_{[0,T]\times\mathbb{R}^d} = \phi\right\}.
	\]
	It imbeds into $C([0,T], H^s) \cap C^1([0,T], H^{s-1})$ for $b>\frac12$.
	
	For the one-dimensional case ($d=1$), to obtain the desired estimates, we work in the smaller function space $X^{s,b}$ introduced by Keel and Tao \cite{KT}, rather than the wave--Sobolev spaces. This space is equipped with the norm
	\[
	\|\phi\|_{X^{s,b}} = \|\Lambda_+^s \Lambda_-^b \phi\|_{L^2(\mathbb{R}^{1+1})}. 
	\]
	Similarly, the time-resticted space of $X^{s,b}$ is defined as
	\[
	X^{s,b}_T = \{ \phi = \Phi|_{[0,T]\times\mathbb{R}} : \Phi \in  X^{s,b}\},
	\]
	with norm
	\[
	\|\phi\|_{X^{s,b}_T} = \inf \left\{\|\Phi\|_{X^{s,b}} :  \Phi|_{[0,T]\times\mathbb{R}} = \phi\right\}.
	\]

	\subsection{Organization of the rest paper}
	In Section~\ref{nf}, we identify the null structure of the systems \eqref{1cso3s1}--\eqref{1cso3s4} and \eqref{cso3s1}--\eqref{cso3s2}. Section~\ref{pre} introduces the corresponding energy estimates and bilinear estimates used in the one- and two-dimensional analyses. Section~\ref{proof} is devoted to the proofs of Theorem~\ref{main thm 1dim} and Theorem~\ref{main thm}. Finally, Section~\ref{appendix} contains a detailed derivation of the coupled wave equations \eqref{1cso3s1}--\eqref{1cso3s4} and \eqref{cso3s1}--\eqref{cso3s2} from the Euler--Lagrange equations, together with the proof of Proposition~\ref{energy1}.
	
	\section{Null structure} \label{nf}
	This section reveals the underlying null structure of both the one- and two-dimensional models, which will play a key role in the proofs of the main theorems.
	
	\subsection{Null forms in the 2D model}\label{nf2d}
	The null forms, first introduced by Klainerman \cite{KM3}, are defined by
	\begin{equation}
		\begin{aligned}
			Q_{0}(f, g) 
			& = \partial^{\alpha} f \partial_{\alpha} g,
			\\
			Q_{\alpha \beta}(u, v) 
			& =\partial_{\alpha} f \partial_{\beta} g-\partial_{\beta} f \partial_{\alpha} g
		\label{Q}
	\end{aligned}
	\end{equation}
	for functions $ f(t, x_1, x_2), g(t, x_1, x_2) $. 
	
	As shown in Jin and Zhang \cite{JZ}, the Chern--Simons gauged $O(3)$ sigma model can be written in this form:
	\[
	\begin{split}
		\square \boldsymbol{\phi} 
		& = -\boldsymbol{\phi} \left( Q_0(\boldsymbol{\phi}, \boldsymbol{\phi}) + 2A_\mu \partial^\mu \boldsymbol{\phi} \cdot (n \times \boldsymbol{\phi}) + A_\mu A^\mu |n \times \boldsymbol{\phi}|^2 \right) - 2A^\mu (n \times \partial_\mu \boldsymbol{\phi}) 
		\\
		& \quad - A_\mu A^\mu n \times (n \times \boldsymbol{\phi}) - \frac{1}{\kappa^2} (\phi_3 \boldsymbol{\phi} - n) (1 - \phi_3)^2 (1 + 2\phi_3), 
		\\
		\square A_\mu 
		& = \frac{1}{\kappa} \epsilon_{\mu \nu \rho} \left( Q^{\nu \rho}(\boldsymbol{\phi}, \boldsymbol{\phi} \times n) + A^\rho \partial^\nu \left( (\boldsymbol{\phi} \times (n \times \boldsymbol{\phi})) \cdot n \right) \right),
	\end{split}
	\]
	Through the Riesz transform $R_i = D^{-1} \partial_i$ together with the divergence-free and curl-free decomposition of the vector field $A_{\mu}$, they derive the null structure of the term $A_{\mu}\partial^{\mu}\boldsymbol{\phi}$.
	
	We expect that the terms $A_{\mu}A^{\mu}$ and $\epsilon_{\mu \nu \rho} A^{\rho} \partial^{\nu} \boldsymbol{\phi}$ can be expressed in the null forms as well. To this end, we employ the technique developed by Huh \cite{Huh1}, who introduced auxiliary vector fields $B_{\mu}$ satisfying
	\begin{equation}\label{eB}
		\begin{aligned}
			\partial_{\mu} B^{\mu} 
			& = 0,
			\\
			\partial_{\mu}B_{\nu} - \partial_{\nu}B_{\mu} 
			& = \epsilon_{\mu \nu \rho}A^{\rho},
		\end{aligned}
	\end{equation}
	with initial data $ B_{\mu}(0,x) = c_{\mu} $ and constraint $ \partial_{1}c_{2} - \partial_{2}c_{1} = a_{0}$. Applying $\partial^{\mu}$ to the second equation in \eqref{eB}, we obtain
	\begin{equation}\label{WaveB}
		\square B_{\nu} = \epsilon_{\mu \nu \rho} \partial^{\mu}A^{\rho},
	\end{equation}
	which is a linear wave equation for $B_{\nu}$ with initial data
	\[
	B_{\nu}(0,x) = c_{\nu}, \quad \partial_{t}B_{\nu}(0,x) = \dot{c}_{\nu}
	\]
	satisfying the constraints
	\begin{equation}
		\begin{aligned}
		\dot{c}_0 = \partial_{1}c_{1} + \partial_{2}c_{2}, \quad
		& \partial_{1}c_{2} - \partial_{2}c_{1} = a_{0},  
		\\
		\dot{c}_1 = \partial_{1}c_{0} - a_{2}, \quad 
		& \dot{c}_2 = \partial_{2}c_{0} - a_{1}.
		\label{eB initial}
	\end{aligned}
	\end{equation}
	The solution to \eqref{WaveB}--\eqref{eB initial} is given by
	\begin{equation}
		\begin{aligned}
			B_{\nu} (t, x) = 
			& \frac{1}{2 \pi t} \int_{\mathcal{B}(t, x)} \frac{[c_{\nu} (y) + D c_{\nu} (y) \cdot (y-x) + t \dot{c}_{\nu} (y)]}{\sqrt{t^2 - | y - x |^2}} dy  
			\\
			& + \frac{1}{2 \pi} \iint_{\mathcal{C}(t, x)} \frac{\epsilon_{\mu \nu \rho} \partial^{\mu}A^{\rho} (\tau, y)}{\sqrt{(t - \tau)^2 - | y - x |^2}} dy d\tau,
			\label{B}
		\end{aligned}
	\end{equation}
	where $ \mathcal{B}(t, x) $ denotes the ball in $ \mathbb{R}^{1+2} $ centered at $x$ with radius $t$, and $\mathcal{C}(t, x) = \{(\tau, y) \in \mathbb{R}^{1+2}: 0 \leq \tau \leq t, |y-x| \leq t-\tau \}$ is a cone in $ \mathbb{R}^{1+2} $.
	
	Conversely, we have
	\begin{equation}\label{Arho}
	A^{\rho} = \epsilon^{\mu \nu \rho} \partial_{\mu} B_{\nu}.
	\end{equation}
	Thus, \eqref{B} and \eqref{Arho} establish a one-to-one correspondence between $A_{\nu}$ and $B_{\nu}$.
	
	Next, making use of \eqref{eB}, Huh \cite{Huh1} gave that
	\begin{equation}\label{nf12}
		\begin{aligned}
			A_{\mu}A^{\mu} 
			&= Q_{0}(B_{\mu}, B^{\mu}) + 2Q_{\mu\nu}(B^{\mu}, B^{\nu}),
			\\
			A^{\mu}\partial_{\mu} (n \times \boldsymbol{\phi}) 
			&=  Q_{10}(B_{2}, n \times \boldsymbol{\phi}) + Q_{02}(B_{1}, n \times \boldsymbol{\phi}) + Q_{21}(B_{0}, n \times \boldsymbol{\phi}).
		\end{aligned}
	\end{equation}
	For the term $\epsilon_{\mu\nu\rho}A^{\rho} \partial_{\nu}\boldsymbol{\phi}$, substituting \eqref{Arho} to this term, we obtain
	\[
	\begin{aligned}
		\epsilon_{\mu \nu \rho}A^{\rho}\partial_{\nu}\boldsymbol{\phi} 
		&= \delta_{\mu}^{\alpha} \delta_{\nu}^{\beta}(\partial_{\alpha}B_{\beta})(\partial^{\nu}\boldsymbol{\phi}) - \delta_{\nu}^{\alpha} \delta_{\mu}^{\beta}(\partial_{\alpha}B_{\beta})(\partial^{\nu}\boldsymbol{\phi}) 
		\\
		&= (\partial_{\mu}B_{\nu})(\partial^{\nu}\boldsymbol{\phi}) - (\partial_{\nu}B_{\mu})(\partial^{\nu}\boldsymbol{\phi}).		
	\end{aligned}
	\]
	Under the Lorenz gauge condition and the definition of null forms \eqref{Q}, we obtain
	\begin{equation}\label{nf3}
		\epsilon_{\mu \nu \rho}A^{\rho}\partial_{\nu}\boldsymbol{\phi} =  Q_{\mu\nu}(B^{\nu}, \boldsymbol{\phi}) - Q_{0}(B_{\mu}, \boldsymbol{\phi}).
	\end{equation}
	Substituting \eqref{nf12}--\eqref{nf3} into \eqref{wave1}--\eqref{wave2} transforms the system into wave equations, with nonlinear derivative terms expressed entirely in terms of null forms:
	\begin{align}
		\begin{split}
			\square \boldsymbol{\phi} = 
			& -\boldsymbol{\phi} Q_{0}(\boldsymbol{\phi}, \boldsymbol{\phi}) - 2\boldsymbol{\phi}(Q_{10}(B_{2}, \boldsymbol{\phi}) + Q_{02}(B_{1}, \boldsymbol{\phi}) + Q_{21}(B_{0}, \boldsymbol{\phi})) \cdot (n \times \boldsymbol{\phi}) 
			\\
			& -(Q_{0}(B_{\mu}, B^{\mu}) + 2Q_{\mu\nu}(B^{\mu}, B^{\nu}))(\boldsymbol{\phi}|n \times \boldsymbol{\phi}|^{2} + n \times (n \times \boldsymbol{\phi})) 
			\\
			&-2(Q_{10}(B_{2}, n \times \boldsymbol{\phi}) + Q_{02}(B_{1}, n \times \boldsymbol{\phi}) + Q_{21}(B_{0}, n \times \boldsymbol{\phi})) 
			\\
			&- \kappa^{-2} (\phi_3 \boldsymbol{\phi} - n) (1 - \phi_3)^2 (1 + 2\phi_3),
			\label{wave1}
		\end{split}
		\\
		\begin{split}
			\Box A_{\mu} = 
			& ~ \kappa^{-1}\epsilon_{\mu \nu \rho} Q^{\nu\rho}(\boldsymbol{\phi}, \boldsymbol{\phi} \times n) + 2\phi_1(Q_{\mu\nu}(\phi_1, B^{\nu}) - Q_{0}(\phi_1, B_{\mu})) 
			\\
			& + 2\phi_2(Q_{\mu\nu}(\phi_2, B^{\nu}) - Q_{0}(\phi_2, B_{\mu})). \label{wave2}
		\end{split}
	\end{align}
	with
	\begin{equation}\label{3ini2}
		\boldsymbol{\phi}(0, \cdot)=\boldsymbol{\phi}_{0}, 
		\quad \partial_{t}\boldsymbol{\phi}(0,\cdot)=\boldsymbol{\phi}_1, 
		\quad A_{\mu}(0, \cdot)=a_{\mu}, 
		\quad \partial_{t} A_{\mu}(0, \cdot)=\dot{a}_{\mu}.
	\end{equation}
	
	\subsection{Null forms in the 1D model} \label{nf1d}
	The null forms in $\mathbb{R}^{1+1}$ are given by:
	\[
	Q_0(f, g) = \partial_0 f \partial_0 g - \partial_1 f \partial_1 g,
	\qquad
	Q_1(f, g) = \partial_0 f \partial_1 g - \partial_1 f \partial_0 g,
	\]
	for functions $ f(t, x_1) $ and $ g(t, x_1) $. Huh and Jin \cite{HJ} derived the wave system with null forms by applying the vector fields $B_{\mu}$ as in the two dimensional case and they calculated:
	\begin{align*}
		& A^\mu \partial_\mu (n \times \phi) = Q_0(B_1, n \times \phi) + Q_1(B_0, n \times \phi), 
		\\
		& A_\mu A^\mu = Q_0(B_1, B_1) - Q_0(B_0, B_0) + 2Q_1(B_0, B_1),
	\end{align*}
	Then \eqref{1cso3s1}--\eqref{1cso3s4} can be rewritten as follows:
	\begin{align}
		\begin{split}
			\Box \boldsymbol{\phi} =
			& - \boldsymbol{\phi} ( Q_0(\boldsymbol{\phi}, \boldsymbol{\phi}) + 2 ( Q_0 (B_1, \boldsymbol{\phi}) + Q_1 (B_0, \boldsymbol{\phi}) ) \cdot (n \times \boldsymbol{\phi}) + |n \times \phi|^2 ( Q_0(B_1, B_1) 
			\\
			&- Q_0(B_0, B_0) + 2Q_1(B_0, B_1) ) + \phi_3 U(\phi_3, N) ) + U(\phi_3, N) n
			\\
			&  - 2 (Q_0 (B_1, n \times \boldsymbol{\phi}) + Q_1 (B_0, n \times \boldsymbol{\phi})) 
			\\
			& + (Q_0(B_1, B_1) - Q_0(B_0, B_0) + 2Q_1(B_0, B_1))(n \times (n \times \boldsymbol{\phi})),
			\label{1wave1}
		\end{split}
		\\
		\Box B_0 = 
		& - \kappa^{-1} N(\phi_1^2 + \phi_2^2), 
		\label{1wave2}
		\\
		\Box B_1 =
		& ~ 0, 
		\label{1wave3}
		\\
		\begin{split}
			\Box N =
			& - \kappa^{-1} \left( 2Q_1(\phi_1, \phi_2) + N(\phi_1^2 + \phi_2^2)^2 + 2\phi_2(Q_1(B_1, \phi_2) + Q_0(B_0, \phi_2)) \right.
			\\
			& \left. + 2\phi_1(Q_1(B_1, \phi_1) + Q_0(B_0, \phi_1))\right) ,
		\label{1wave4}
		\end{split}
	\end{align}
	with the initial data
	\begin{align}
		\begin{split}
			& \boldsymbol{\phi}(0, \cdot) = \boldsymbol{\phi}_0, \quad \partial_t \boldsymbol{\phi}(0, \cdot) = \boldsymbol{\phi}_1, 
			\\
			& B_\mu(0, \cdot) = 0, \quad \partial_t B_0(0, \cdot) = - a_1, \quad \partial_t B_1(0, \cdot) = a_0, 
			\\
			& N(0, \cdot) = n_0, \quad \partial_t N(0, \cdot) = n_1,
		\end{split}
		\label{3ini}
	\end{align}
	satisfying the constraint $\partial_1 n_0 + \langle n \times \boldsymbol{\phi}_0, \boldsymbol{\phi}_1 + a_0(n \times \boldsymbol{\phi}_0) \rangle = 0$, and $n_1 = - \langle n \times \boldsymbol{\phi}_0, \partial_1 \boldsymbol{\phi}_0 + a_1(n \times \boldsymbol{\phi}_0) \rangle$.

	\section{Preliminary} \label{pre}
	
	In this section, we collect the fundamental linear and nonlinear estimates that form the analytic core of our local well-posedness arguments. For the two-dimensional case, we work in the wave--Sobolev $\mathcal{H}^{s,b}$ for $s>1, b>\frac12$; for the one-dimensional case, we utilize the specialized $X^{s,b}$ spaces for $s>b>\frac12$. We will first present the linear estimates, followed by the crucial bilinear product estimates.

	\subsection{Estimates for the 2D case}\label{pre2d}
	The following lemma provides the crucial linear energy estimates.
	
	\begin{lemma}[\cite{Selberg}, Theorem~13]\label{energy2}
		Assume $s \in \mathbb{R}, b \in\left(\frac{1}{2}, 1\right), \epsilon \in[0,1-b]$. Consider the Cauchy problem for the linear wave equation
		\begin{equation}\label{linearw}
			\begin{cases}
				\square w=F(t,x), \quad(t, x) \in \mathbb{R}^{1+d}, 
				\\
				\left.w\right|_{t=0}=f,\left.\quad \partial_t w\right|_{t=0}=g,
			\end{cases}
		\end{equation}
		let $f, g$ and $F$ satisfy $f \in H^{s}, g \in H^{s-1}$, and $F \in H^{s-1, b+\epsilon-1}$.
		
		Let $0<T<1$ and define
		\begin{equation}\label{defw}
			w(t)=\theta(t)w_0 + \theta_T(w_1+ w_2),
		\end{equation}
		where
		\begin{equation}\label{defw012}
			\begin{aligned}
				& w_0=\cos (t D) f+D^{-1} \sin (t D) g ,
				\\
				& F_{1}=\eta\left(T^{\frac14} \Lambda_{-}\right) F, \quad F_{2}=\left(1-\eta\left(T^{\frac14} \Lambda_{-}\right)\right) F,
				\\
				& w_1= - \int_{0}^{t} D^{-1} \sin \left(\left(t-t^{\prime}\right) D\right) F_{1}\left(t^{\prime}\right) d t^{\prime} , 
				\\
				& w_2=\square^{-1} F_2,
			\end{aligned}
		\end{equation}
		Then, the function $w$ defined in \eqref{defw}--\eqref{defw012} is the unique solution to \eqref{linearw} on $[0, T] \times \mathbb{R}^d$ such that $w \in C([0, T] ; H^{s}) \cap C^1([0, T] ; H^{s-1})$ and satisfies the following estimate:
		\begin{equation}\label{energy22}
			|w|_{s, b} \leq C_{0}\left( \|f\|_{H^{s}}+\|g\|_{H^{s-1}}+T^{\frac{\epsilon}{4}}\|F\|_{s-1, b+\epsilon-1}\right) ,
		\end{equation}
		where $C_{0}$ only depends on $\theta$ and $b$.
	\end{lemma}
	
	\begin{remark}
		For $\epsilon=0$ or $\epsilon \in(0,1-b]$, we refer the reader to Selberg's paper \cite{Selberg} Theorem 12 and Theorem 13 respectively. 
	\end{remark}
	
	\begin{remark}\label{pe}
		We emphasize that while the linear estimate is performed in the global wave--Sobolev spaces $\mathcal{H}^{s,b}$ over $\mathbb{R}^{1+d}$, the actual local-in-time solution to the Cauchy problem \eqref{linearw} is well-posed in the time-restricted spaces $\mathcal{H}_T^{s,b}$. By Selberg's embedding lemma, the solutions obtained in these restricted spaces naturally belong to $C([0,T]; H^s) \cap C^1([0,T]; H^{s-1})$ for $b > \frac{1}{2}$. Notably, the time localization extracts a small parameter $T^{\epsilon}$, which acts as the crucial contraction factor in our subsequent Picard iteration scheme.
	\end{remark}

	To bound the nonlinear terms appearing in the right-hand side of the linear estimates, we require bilinear embeddings of the form $H^{s_1,b_1} \cdot H^{s_2,b_2} \hookrightarrow H^{-s_0,-b_0}$, which imply
	\[
	\|f g\|_{-s_{0},-b_{0}} \leq C \|f\|_{s_{1}, b_{1}}\|g\|_{s_{2}, b_{2}} \text { for all } f, g \in \mathcal{S}(\mathbb{R}^{1+d}),
	\]
	where $C$ depends on the indices $s_{\alpha}, b_{\alpha}$ and the dimension $d$. We say that the exponent matrix
	\[
	\left(\begin{array}{lll}
	s_{0} & s_{1} & s_{2} 
	\\
	b_{0} & b_{1} & b_{2}
	\end{array}\right)
	\]
	is a product if inequality holds. The following fundamental product estimates in $H^{s,b}(\mathbb{R}^{1+d})$, for $d=1,2$, were established in \cite{AFS1}.
	
	\begin{lemma}[\cite{AFS1}, Theorem~5.1, Theorem~7.1]\label{product}
		Assume
		\begin{align*}
			& b_{0} + b_{1}+b_{2} > \frac{1}{2},
			\\
			& b_{0} + b_{1} \geq 0, 
			\\
			& b_{1} + b_{2} \geq 0,
			\\
			& b_{0} + b_{2} \geq 0,
			\\
			& s_{0} + s_{1} + s_{2} > \frac{d+1}{2} - (b_{0}+b_{1}+b_{2}),
			\\
			& s_{0} + s_{1} + s_{2} > \frac{d}{2} - \min (b_{0}+b_{1}, b_{0}+b_{2}, b_{1}+b_{2}),
			\\
			& s_{0} + s_{1} + s_{2} > \frac{d-1}{2} - \min (b_{0}, b_{1}, b_{2}), 
			\\
			& s_{0}+s_{1} + s_{2} > \frac{d+1}{4}, 
			\\
			& (s_{0}+b_{0}) + 2 s_{1} + 2 s_{2} > \frac{d}{2},
			\\
			& 2 s_{0} + (s_{1} + b_{1}) + 2 s_{2} > \frac{d}{2},
			\\
			& 2 s_{0} + 2 s_{1} + (s_{2}+b_{2}) > \frac{d}{2},
			\\
			& s_{0} + s_{1} \geq \max (0,-b_{2}), 
			\\
			& s_{1} + s_{2} \geq \max (0,-b_{0}), 
			\\
			& s_{0} + s_{2} \geq \max (0,-b_{1}).
		\end{align*}
		Then
		\[
		\left(\begin{array}{lll}
		s_{0} & s_{1} & s_{2} 
		\\
		b_{0} & b_{1} & b_{2}
		\end{array}\right)
		\]
		is a product.
	\end{lemma}
		
	\subsection{Estimates for the 1D case}
	
	We now turn to the one-dimensional case. The proof of the one-dimensional well-posedness result of Huh and Jin~\cite{HJ} relies crucially on the product estimate stated in Lemma~\ref{product} for $d=1$. However, this product estimate is insufficient to establish the bilinear estimates required at the lower regularity threshold $s>\frac12$. To overcome this difficulty, we adopt the null-coordinate framework of Keel and Tao~\cite{KT} in the one-dimensional case.
	
	Following~\cite{KT}, we work in the space $X^{s,b}$. This choice is dictated by the nonlinear analysis. Indeed, the product estimates available in the null-coordinate framework apply directly to the null-form nonlinearities arising in our system, thereby yielding the bilinear estimates needed at the regularity threshold $s>\frac12$. 
	
	Our first ingredient is a new linear energy estimate adapted to $X^{s,b}$. Unlike the approach of Keel and Tao~\cite{KT}, which first reduces the problem to the small-data regime via finite propagation speed, we derive this estimate directly in $X^{s,b}$.
	
	\begin{proposition} \label{energy1}
		Assume $s > \frac12, b \in\left(\frac{1}{2}, 1\right), \epsilon \in[0,1-b]$, let $f, g$ and $F$ satisfy $f \in H^{s}, g \in H^{s-1}$, and $F \in X^{s-1, b+\epsilon-1}$, and \eqref{defw}--\eqref{defw012} hold, then $w$ is the unique to \eqref{linearw} on $[0, T] \times \mathbb{R}^d$ such that $w \in C([0, T] ; H^{s}) \cap C^1([0, T] ; H^{s-1})$ and satisfies the following estimate:
		\begin{equation}\label{energy12}
			\|w\|_{X^{s, b}} \leq C_{0}(\|f\|_{H^{s}}+\|g\|_{H^{s-1}}+T^{\frac{\epsilon}{4}}\|F\|_{X^{s-1, b+\epsilon-1}}),
		\end{equation}
		where $C_{0}$ only depends on the cut-off function $\theta$ and $b$.
	\end{proposition}
	
	\begin{remark}
		An important feature of Proposition~\ref{energy1} is that the estimate is established without relying on finite propagation speed. This estimate provides the linear ingredient required for carrying out the Picard iteration entirely within the $X^{s,b}$ framework. We defer the proof to Appendix~\ref{proof 3.1}, where it is obtained by a minor modification of Theorem~12 in Selberg~\cite{Selberg}.
	\end{remark}
	
	\begin{remark}
		We emphasize that while the estimate is performed in the global $X^{s,b}$ over $\mathbb{R}^{1+d}$, the actual local-in-time solution to the Cauchy problem \eqref{linearw} is well-posed in the time-restricted spaces $X^{s,b}_T$.
	\end{remark}

	To estimate the nonlinear terms, we introduce the standard null coordinates:
	$$ 
	u = x+t, \quad v = x-t. 
	$$
	Under this change of variables, the derivatives transform as $\partial_x = \partial_u + \partial_v$ and $\partial_t = \partial_u - \partial_v$. The wave operator factors as $\square = -4\partial_u \partial_v$, and the standard null forms decouple into terms:
	\begin{align*}
		Q_0(f,g) 
		&= \partial_t f \partial_t g - \partial_x f \partial_x g = -2(\partial_u f \partial_v g + \partial_v f \partial_u g), 
		\\
		Q_1(f,g) 
		&= \partial_t f \partial_x g - \partial_x f \partial_t g = 2(\partial_u f \partial_v g - \partial_v f \partial_u g).
	\end{align*}
	Taking advantage of the null coordinates, we can rewrite the $X^{s,b}$ norms in terms of product Sobolev spaces $ H_u^{s_1}H_v^{s_2} = H_v^{s_2}H_u^{s_1} $ defined by
	\[
	\| \boldsymbol{\phi} \|_{H_u^{s_1}H_v^{s_2}} = \| \Lambda_u^{s_1} \Lambda_v^{s_2} \boldsymbol{\phi} \|_{L^2(\mathbb{R}^{1+1})},
	\]
	where \( \Lambda_u \) and \( \Lambda_v \) are the Fourier multipliers corresponding to \( \langle \mu \rangle \), \( \langle \nu \rangle \), respectively, and \( \mu, \nu \) are the frequency variables dual to \( u, v \). We define the one-dimensional Sobolev spaces \( H_u^s, H_v^s \) in the usual manner. By Plancherel’s theorem, one can easily verify that
	\begin{equation}\label{X}
		X^{s,b} = H_u^s H_v^b \cap H_v^s H_u^b
	\end{equation}
	when \( b \leq s \). Thus, bounding the null forms in the $X^{s,b}$ framework reduces to establishing product estimates in these $H^{s_1}_u H^{s_2}_v$ spaces. For this purpose, we rely on the following fundamental estimates established by Keel and Tao \cite{KT}.
	\begin{lemma}[\cite{KT}, Lemma~3.3]\label{uvproduct}
		If \( s_1, s_2 > 1/2 \), and \( s_1 \geq s_1' \geq -s_1 \), \( s_2 \geq s_2' \geq -s_2 \), then 
		\begin{align*}
			\|f g\|_{H_u^{s_1^{\prime}} H_v^{s_2^{\prime}}} 
			& \lesssim \|f\|_{H_u^{s_1} H_v^{s_2}} \|g\|_{H_u^{s_1^{\prime}} H_v^{s_2^{\prime}}},
			\\
			\|f g\|_{H_u^{s_1^{\prime}} H_v^{s_2^{\prime}}} 
			& \lesssim \|f\|_{H_u^{s_1^{\prime}} H_v^{s_2}} \|g\|_{H_u^{s_1} H_v^{s_2^{\prime}}}
		\end{align*}
		for all test functions \( f, g \).
	\end{lemma}

	\section{Proof of the main results} \label{proof}
	
	In this section, we prove the local well-posedness results by applying the contraction mapping principle. The overall strategy is identical for both spatial dimensions: we first construct the solution map associated with the linear wave equations and then establish null-form estimates. We address the 2D case in Section~\ref{proof2d}, followed by the 1D case in Section~\ref{proof1d}.

	\subsection{Local well-posedness for 2D} \label{proof2d}
	
	We now turn to the proof of Theorem~\ref{main thm}. The argument is structured as follows: we first formulate the Picard iteration map $\mathbf{M}$ in an appropriate resolution space, then systematically bound the null forms to establish the core nonlinear estimates for the nonlinearities $\mathbf{G}$ and $H_{\mu}$, and finally demonstrate that $\mathbf{M}$ is a strict contraction for a sufficiently small lifespan $T>0$. 
	
	To begin with, we define the solution space
	\begin{equation}\label{Ls}
		\begin{aligned}
			L_{s, b}=\{(\boldsymbol{\phi}, \mathbf{A}) \in \mathcal{H}^{s,b}(\mathbb{R}^{1+2}) & \times \mathcal{H}^{s-\frac34,b}(\mathbb{R}^{1+2}): 
			| \boldsymbol{\phi}|_{s, b}+|\mathbf{A}|_{s-\frac{3}{4}, b} \leq 2\left(1+C_{0}\right) M_{0}
			\\
			& (\boldsymbol{\phi}, \partial_{t} \boldsymbol{\phi})|_{t=0}=(\boldsymbol{\phi}_{0}, \boldsymbol{\phi}_{1}),(\mathbf{A}, \partial_{t} \mathbf{A})|_{t=0}=(\mathbf{a}, \dot{\mathbf{a}})\},
		\end{aligned}
	\end{equation}
	where
	\[
	M_{0}=\left\|\boldsymbol{\phi}_{0}\right\|_{H^{s}}+\left\|\boldsymbol{\phi}_{1}\right\|_{H^{s-1}}+\|\mathbf{a}\|_{H^{s-\frac{3}{4}}}+\left\|\dot{\mathbf{a}}\right\|_{H^{s-\frac{7}{4}}}
	\]
	and $C_{0}$ is the same positive constant as in \eqref{energy22}. For $(\boldsymbol{\phi}, \mathbf{A}) \in L_{s, b}$, following the time-localization procedure introduced in \eqref{defw012}, we denote the map $\mathbf{M}$ by
	\begin{equation}\label{dM}
		\begin{aligned}
			\mathbf{M} \boldsymbol{\phi} = ~
			& \theta(t) \Phi_0 + \theta_T(t)( \Phi_1 + \Phi_2),
			\\
			\mathbf{M} A_{\mu} = ~
			& \theta(t) \mathcal{A}_{0\mu} + \theta_T(t)(\mathcal{A}_{1\mu} + \mathcal{A}_{2\mu}),
		\end{aligned}
	\end{equation}
	where $\Phi_0, \mathcal{A}_{0\mu}$ are the linear evolutions of the initial data, $\Phi_1,\Phi_2$ and $\mathcal{A}_{1\mu},\mathcal{A}_{2\mu}$ correspond to the nonlinear terms of \eqref{wave1} and \eqref{wave2},  denote by $\mathbf{G}$ and $H_{\mu}$, respectively. The auxiliary fields $B_{\mu}$ obey the regularity bound $|B_{\mu}|_{s+\frac14,b} \lesssim |A_{\mu}|_{s-\frac34,b}$. By construction, the mapped variables satisfy the localized wave equations:
	\begin{equation}\label{eM1}
		\begin{cases}
			\square \mathbf{M} \boldsymbol{\phi}=\mathbf{G}, \quad(t, x) \in [0, T] \times \mathbb{R}^2,
			\\
			\mathbf{M} \boldsymbol{\phi}|_{t=0}=\boldsymbol{\phi}_0, \quad \partial_t \mathbf{M} \boldsymbol{\phi}|_{t=0}=\boldsymbol{\phi}_1
		\end{cases}
	\end{equation}
	and
	\begin{equation}\label{eM2}
		\begin{cases}
			\square \mathbf{M} A_{\mu}=H_{\mu}, 
			& (t, x) \in[0, T] \times \mathbb{R}^2 ,
			\\
			\left.\mathbf{M} A_{\mu}\right|_{t=0}=a_{\mu}, 
			& \left.\partial_t \mathbf{M} A_{\mu}\right|_{t=0}=\dot{a}_{\mu}.
		\end{cases}
	\end{equation}
	So any fixed point of \( \mathbf{M} \) is a solution of the problem \eqref{cso3s1}--\eqref{2ini2} on \( [0,T] \times \mathbb{R}^2 \). Set $\mathbf{H}=\left(H_0, H_1, H_2\right)$. For $s>1$, $b \in\left(\frac{1}{2}, 1\right), \epsilon \in(0,1-b]$, Lemma~\ref{energy2} yields
	\begin{equation}\label{ephiA}
		\begin{aligned}
			& |\mathbf{M} \boldsymbol{\phi}|_{s, b} \leq C_0\left(\left\|\boldsymbol{\phi}_{0}\right\|_{H^{s}}+\left\|\boldsymbol{\phi}_1\right\|_{H^{s-1}}+T^{\frac{\epsilon}{4}}\|\mathbf{G}\|_{s-1, b+\epsilon-1}\right) ,
			\\
			& |\mathbf{M A}|_{s, b} \leq C_0\left(\|\mathbf{a}\|_{H^{s-\frac{3}{4}}}+\left\|\dot{\mathbf{a}}\right\|_{H^{s-\frac{7}{4}}}+T^{\frac{\epsilon}{4}}\|\mathbf{H}\|_{s-\frac{7}{4}, b+\epsilon-1}\right), \quad \mu=0,1,2.
		\end{aligned}
	\end{equation}
	Hence the proof reduces to estimating the nonlinear terms $\|\mathbf{G}(\boldsymbol{\phi}, \mathbf{B})\|_{s-1, b+\epsilon-1}$ and $\|H_{\mu}(\boldsymbol{\phi}, \mathbf{B})\|_{s-\frac{7}{4}, b+\epsilon-1}$ for $\boldsymbol{\phi} \in \mathcal{H}^{s,b}$ and $\mathbf{B} \in \mathcal{H}^{s+\frac14,b}$. We will repeatedly use Lemma~7.6 and Lemma~8.1 of \cite{KS}, which reduce null-form estimates to bilinear product estimates in $H^{s,b}$.
	
	{\bfseries Step 1: estimate for $Q_0(\boldsymbol{\phi}, \boldsymbol{\phi})$.}
	
	We first estimate the null form  $Q_0(\boldsymbol{\phi}, \boldsymbol{\phi})$. Since, by Lemma~\ref{product}, we have
	\[
	\|\boldsymbol{\phi} Q_{0}(\boldsymbol{\phi}, \boldsymbol{\phi})\|_{s-1, b+\epsilon-1} \lesssim \|\boldsymbol{\phi}\|_{s, b} \|Q_{0}(\boldsymbol{\phi}, \boldsymbol{\phi})\|_{s-1, b+\epsilon-1}.
	\]
	Thus, it suffices to show that $Q_0(\mathcal{H}^{s,b}, \mathcal{H}^{s,b}) \subseteq H^{s-1,b+\epsilon-1}$. By Lemma 7.6 in \cite{KS} and the fractional Leibniz rule, we have
	\[
	\begin{aligned}
		Q_0(\boldsymbol{\phi},\boldsymbol{\phi}) 
		& \precsim D_+^{1-\epsilon} D_-^{1-\epsilon} (D_+^\epsilon \boldsymbol{\phi} \cdot D_+^\epsilon \boldsymbol{\phi}) + (D_+ D_-^{1-\epsilon} \boldsymbol{\phi}) \cdot D_+^\epsilon \boldsymbol{\phi} + D_+^\epsilon \boldsymbol{\phi} \cdot (D_+ D_-^{1-\epsilon} \boldsymbol{\phi})
		\\
		& \precsim D_-^{1-\epsilon} (D_+\boldsymbol{\phi} \cdot D_+^\epsilon \boldsymbol{\phi} ) + (D_+ D_-^{1-\epsilon} \boldsymbol{\phi}) \cdot D_+^\epsilon \boldsymbol{\phi}.
	\end{aligned}
	\]
	Since for $\boldsymbol{\phi} \in \mathcal{H}^{s,b}$, we have
	\begin{align*}
		\| D_+^\epsilon \boldsymbol{\phi} \|_{s-\epsilon,b} 
		& \lesssim \| \langle \xi \rangle ^{s-\epsilon} \langle |\tau|-|\xi| \rangle^b (|\tau|+|\xi|)^\epsilon \widetilde{\boldsymbol{\phi}} \|_{L^2}
		\\
		& \lesssim \| \frac{\langle \xi \rangle ^{s-\epsilon} (|\tau|+|\xi|)^\epsilon}{\langle \xi \rangle ^{s-1} \langle |\tau|+|\xi| \rangle} \langle \xi \rangle ^{s-1} \langle|\tau|+|\xi|\rangle \langle |\tau|-|\xi| \rangle^b  \widetilde{\boldsymbol{\phi}} \|_{L^2}
		\\
		& \lesssim \| \langle \xi \rangle ^{s-1} \langle|\tau|+|\xi|\rangle \langle |\tau|-|\xi| \rangle^b  \widetilde{\boldsymbol{\phi}} \|_{L^2}
		\\
		& \lesssim \| \Lambda^{s-1} \Lambda_{+} \Lambda_{-}^b \boldsymbol{\phi}\|_{L^2} \lesssim |\boldsymbol{\phi}|_{s,b}.
	\end{align*}
	Hence $D_+^\epsilon \boldsymbol{\phi} \in H^{s-\epsilon,b}$. Consequently, it suffices to verify the products
	\begin{align*}
		D_-^{1-\epsilon} (H^{s-1,b} \cdot H^{s-\epsilon,b})
		& \hookrightarrow H^{s-1,b+\epsilon-1},
		\\
		H^{s-1,b+\epsilon-1} \cdot H^{s-\epsilon,b}
		& \hookrightarrow H^{s-1,b+\epsilon-1}.
	\end{align*}
	For the first, it is enough to check
	\[
	H^{s-1,b} \cdot H^{s-\epsilon,b} \hookrightarrow H^{s-1,b},
	\]
	due to the fact that
	\begin{align*}
		\| D_-^{1-\epsilon} \boldsymbol{\phi} \|_{s-1,b+\epsilon-1} 
		& \lesssim \| \langle \xi \rangle ^{s-1} \langle |\tau|-|\xi| \rangle^{b+\epsilon-1} (||\tau| - |\xi||)^{1-\epsilon} \widetilde{\boldsymbol{\phi}} \|_{L^2}
		\\
		& \lesssim \| \Lambda^{s-1} \Lambda_{-}^b \boldsymbol{\phi}\|_{L^2} 
		\lesssim \|\boldsymbol{\phi}\|_{s-1,b},
	\end{align*}
	i.e. $D_-^{1-\epsilon} H^{s-1,b} \hookrightarrow H^{s-1,b+\epsilon-1}$. By Lemma~\ref{product}, we can verify these imbeddings are valid.

	{\bfseries Step 2: estimate for $Q_{\alpha \beta}(B_{\mu}, \boldsymbol{\phi})$.}
	
	We next estimate the null forms involving the gauge field. By Lemma~\ref{product}, we obtain
	\[
	\|\boldsymbol{\phi}(Q_{\alpha \beta}(B_{\mu}, \boldsymbol{\phi}) \cdot (n \times \boldsymbol{\phi})) \|_{s-1, b+\epsilon-1} \lesssim \|\boldsymbol{\phi}\|_{s,b}^2 \| Q(B_{\mu}, \boldsymbol{\phi}) \|_{s-1, b+\epsilon-1}.
	\]
	Hence it suffices to show  $Q_{\alpha \beta}(\mathcal{H}^{s+\frac14,b}, \mathcal{H}^{s,b}) \subseteq H^{s-1,b+\epsilon-1}$. By Lemma 8.1 of \cite{KS} and the fact that $B_{\mu} \in \mathcal{H}^{s+\frac{1}{4},b}$, we have
	\begin{align*}
		Q_{\alpha \beta}(B_{\mu},\boldsymbol{\phi}) 
		&\precsim D_-^{\frac{1}{2}} (D_+^{\frac{1}{2}} B_{\mu} \cdot  D_+\boldsymbol{\phi}) + D_-^{\frac{1}{2}} (D_+ B_{\mu} \cdot D_+^{\frac{1}{2}} \boldsymbol{\phi})
		\\
		&+ D_+^{\frac{1}{2}}B_{\mu} \cdot D_-^{\frac{1}{2}} D_+\boldsymbol{\phi} + D_+ B_{\mu} \cdot D_+^{\frac{1}{2}} D_-^{\frac{1}{2}} \boldsymbol{\phi} + D_+ D_-^{\frac{1}{2}} B_{\mu} \cdot D_+^{\frac{1}{2}} \boldsymbol{\phi} + D_+^{\frac{1}{2}} D_-^{\frac{1}{2}} B_{\mu} \cdot D_+\boldsymbol{\phi}.
	\end{align*}
	Thus the desired estimate follows from the product embeddings
	\begin{align*}
		H^{s-\frac{3}{4},b} \cdot H^{s-\frac{1}{2},b}
		& \hookrightarrow H^{s-1,b+\epsilon-\frac{1}{2}},
		\\
		H^{s-\frac{1}{4},b} \cdot H^{s-1,b}
		& \hookrightarrow H^{s-1,b+\epsilon-\frac{1}{2}}
	\end{align*}
	and
	\begin{align*}
		H^{s-\frac{1}{4},b} \cdot H^{s-1,b-\frac{1}{2}}
		& \hookrightarrow H^{s-1,b+\epsilon-1},
		\\
		H^{s-\frac{3}{4},b} \cdot H^{s-\frac{1}{2},b-\frac{1}{2}}
		& \hookrightarrow H^{s-1,b+\epsilon-1},
		\\
		H^{s-\frac{3}{4},b-\frac{1}{2}} \cdot H^{s-\frac{1}{2},b}
		& \hookrightarrow H^{s-1,b+\epsilon-1},
		\\
		H^{s-\frac{1}{4},b-\frac{1}{2}} \cdot H^{s-1,b}
		& \hookrightarrow H^{s-1,b+\epsilon-1}.
	\end{align*}
	All these estimates follow from Lemma~\ref{product}.

	{\bfseries Step 3: estimate for $Q_0(B_{\mu}, B^{\mu})$ and $Q_{\mu \nu}(B^{\mu}, B^{\nu})$.}
	
	By Lemma~\ref{product}, we obtain
	\[
	\| Q_{0}(B_{\mu}, B^{\mu})(\boldsymbol{\phi} |n \times \boldsymbol{\phi}|^2 + n \times (n \times \boldsymbol{\phi}))\|_{s-1, b+\epsilon-1} \lesssim (\|\boldsymbol{\phi}\|_{s,b}^3 + \|\boldsymbol{\phi}\|_{s,b}) \|Q_{0}(B_{\mu}, B^{\mu})\|_{s-1, b+\epsilon-1}.
	\]
	We can prove
	\[
	\|Q_{0}(B_{\mu}, B^{\mu})\|_{s-1, b+\epsilon-1} \lesssim  |\mathbf{B}|_{s+\frac{1}{4},b}^2,
	\]
	which follows from the product embeddings
	\begin{align*}
		H^{s-\frac{3}{4},b} \cdot H^{s+\frac{1}{4}-\epsilon,b}
		& \hookrightarrow H^{s-1,b},
		\\
		H^{s-\frac{3}{4},b-\epsilon+1} \cdot H^{s+\frac{1}{4}-\epsilon,b}
		& \hookrightarrow H^{s-1,b+\epsilon-1}.
	\end{align*}
	The estimate for $\|Q_{\mu \nu}(B^{\mu}, B^{\nu})\|_{s-1, b+\epsilon-1}$ reduces to
	\begin{align*}
		H^{s-\frac{3}{4},b} \cdot H^{s-\frac{1}{4},b}
		& \hookrightarrow H^{s-1,b+\epsilon-\frac{1}{2}},
		\\
		H^{s-\frac{3}{4},b-\frac{1}{2}} \cdot H^{s-\frac{1}{4},b}
		& \hookrightarrow H^{s-1,b+\epsilon-1},
		\\
		H^{s-\frac{1}{4},b-\frac{1}{2}} \cdot H^{s-\frac{3}{4},b}
		& \hookrightarrow H^{s-1,b+\epsilon-1}.
	\end{align*}
	These embeddings follow from Lemma~\ref{product}.

	{\bfseries Step 4: estimate for the remaining polynomial term.}
	
	The remaining polynomial nonlinearities contain no derivatives. Hence they are estimated directly by Lemma~\ref{product}, yielding
	\[
	\|\frac{1}{\kappa^{2}}\left(\boldsymbol{\phi}\left(n \cdot \boldsymbol{\phi}\right)-n(\boldsymbol{\phi} \cdot \boldsymbol{\phi})\right)\left(1-n \cdot \boldsymbol{\phi}\right)^{2}\left(1+2 n \cdot \boldsymbol{\phi}\right)\|_{s-1, b+\epsilon-1} \lesssim \left(|\boldsymbol{\phi}|_{s, b}^{2}+|\boldsymbol{\phi}|_{s, b}^{5}\right).
	\]
	Thus we complete the proof of $\|\mathbf{G}(\boldsymbol{\phi}, \mathbf{B})\|_{s-1, b+\epsilon-1} $.
	
	{\bfseries Step 5: estimate for $H_{\mu}$.}
	
	Finally, the estimates for \(H_\mu\) are completely analogous. Indeed, every derivative term in \(H_\mu\) is again expressed through \(Q_0\) or \(Q_{\alpha\beta}\), the only difference lies in the Sobolev index \( s - \frac{7}{4} \), which is handled by the same product estimates. Therefore the same argument as above gives
	\begin{align*}
		& \| Q_{\alpha \beta}(\boldsymbol{\phi}, \boldsymbol{\phi} \times n)\|_{s-\frac{7}{4}, b+\epsilon-1}
		\lesssim |\boldsymbol{\phi}|^2_{s,b},
		\\
		& \|\boldsymbol{\phi} Q_{\alpha \beta}(\boldsymbol{\phi}, 
		B^{\nu})\|_{s-\frac{7}{4}, b+\epsilon-1} \lesssim \|\boldsymbol{\phi}\|_{s,b} \|Q(\boldsymbol{\phi}, B^{\nu})\|_{s-\frac{7}{4}, b+\epsilon-1}
		\lesssim |\boldsymbol{\phi}|_{s,b}^2 |\mathbf{B}|_{s+\frac{1}{4},b},
		\\
		& \|\boldsymbol{\phi} Q_0(\boldsymbol{\phi}, B_{\mu})\|_{s-\frac{7}{4}, b+\epsilon-1} \lesssim \|\boldsymbol{\phi}\|_{s,b} \|Q_0(\boldsymbol{\phi}, B_{\mu}\|_{s-\frac{7}{4}, b+\epsilon-1}
		\lesssim |\boldsymbol{\phi}|_{s,b}^2 |\mathbf{B}|_{s+\frac{1}{4},b}.
	\end{align*}
	Hence $\|H_{\mu}(\boldsymbol{\phi}, \mathbf{B})\|_{s-\frac{7}{4}, b+\epsilon-1}$ follows.
	
	Collecting the estimates obtained in Steps~1--5, we conclude that
	\begin{equation}\label{esum}
		\begin{aligned}
			& \|\mathbf{G}\|_{s-1, b+\epsilon-1} + \|\mathbf{H}\|_{s-\frac{7}{4}, b+\epsilon-1}
			\\
			& \leq  C \left(|\boldsymbol{\phi}|_{s, b}^{2} + | \boldsymbol{\phi}|_{s, b}^{5} + |\boldsymbol{\phi}|_{s,b}^3 | \mathbf{B}|_{s+\frac{1}{4},b}^2 + | \boldsymbol{\phi}|_{s,b}^3 |\mathbf{B}|_{s+\frac{1}{4},b}  \right.
			\\
			& \left. + |\boldsymbol{\phi}|_{s,b}^2 | \mathbf{B}|_{s+\frac{1}{4},b} + |\boldsymbol{\phi}|_{s,b} | \mathbf{B}|_{s+\frac{1}{4},b}^2 + |\boldsymbol{\phi}|_{s,b} |\mathbf{B}|_{s+\frac{1}{4},b} \right).
		\end{aligned}
	\end{equation}
	
	Substituting these nonlinear estimates into the linear estimates \eqref{ephiA}, and choosing
	\begin{equation}\label{T}
		T=\left(\frac{M_0}{C[2(1+C_0)M_0]^5}\right)^{\frac{4}{\epsilon}},
	\end{equation}
	we finally obtain
	\begin{equation}\label{Me}
		|\mathbf{M} \boldsymbol{\phi}|_{s, b}+|\mathbf{M A}|_{s-\frac{3}{4}, b} \leq 2 C_{0} M_{0},
	\end{equation}
	hence $\mathbf{M}$ maps $L_{s, b}$ to itself. Next we prove that it is a contraction map on $L_{s, b}$.
	
	Since \eqref{eM1}--\eqref{eM2} is a semilinear system, for any $(\boldsymbol{\phi}, \mathbf{A}),(\boldsymbol{\psi}, \mathbf{D}) \in L_{s, b}$, we have
	\[
	\begin{aligned}
		& \|\mathbf{G}(\boldsymbol{\phi},\mathbf{B}) - \mathbf{G}(\boldsymbol{\psi},\mathbf{D})\|_{s-1, b+\epsilon-1}
		+ \|\mathbf{H}(\boldsymbol{\phi},\mathbf{B}) - \mathbf{H}(\boldsymbol{\psi},\mathbf{D})\|_{s-\frac{7}{4}, b+\epsilon-1}
		\\
		& \lesssim P\left( |\boldsymbol{\phi}|_{s,b}, |\boldsymbol{\psi}|_{s,b}, |\mathbf{A}|_{s-\frac{3}{4},b}, |\mathbf{D}|_{s-\frac{3}{4},b} \right)  \left( |\boldsymbol{\phi} -  \boldsymbol{\psi}|_{s, b} + |\mathbf{A} -  \mathbf{D}|_{s-\frac{3}{4}, b}\right),
	\end{aligned}
	\]
	where \(P\) is a polynomial satisfying \(P(0)=0\). This estimate follows directly from \eqref{esum} together with the identities $Q_{\alpha \beta}(\phi,\psi)-Q_{\alpha \beta}(\phi^{\star},\psi^{\star}) = Q_{\alpha \beta}(\phi-\phi^{\star},\psi)+Q_{\alpha \beta}(\phi^{\star},\psi-\psi^{\star})$, and the analogous identity for $Q_0$. Consequently,
	\[
	\begin{aligned}
		& |\mathbf{M} \boldsymbol{\phi}-\mathbf{M} \boldsymbol{\psi}|_{s, b}
		+|\mathbf{M} \mathbf{A}-\mathbf{M} \mathbf{D}|_{s-\frac{3}{4}, b} 
		\\
		& \leq  C_0 T^{\frac{\epsilon}{4}} P\left( |\boldsymbol{\phi}|_{s,b}, |\boldsymbol{\psi}|_{s,b}, |\mathbf{A}|_{s-\frac{3}{4},b}, |\mathbf{D}|_{s-\frac{3}{4},b} \right) \left( |\boldsymbol{\phi} -  \boldsymbol{\psi}|_{s, b} + |\mathbf{A} -  \mathbf{D}|_{s-\frac{3}{4}, b} \right) .
	\end{aligned}
	\]
	For sufficiently small $T$, this implies
	\begin{equation}\label{contraction}
		|\mathbf{M} \boldsymbol{\phi} - \mathbf{M} \boldsymbol{\psi}|_{s,b} + |\mathbf{M} \mathbf{A} - \mathbf{M D}|_{s-\frac{3}{4}, b} \leq \frac{1}{2}\left(|\boldsymbol{\phi} - \boldsymbol{\psi}|_{s,b} + |\mathbf{A} - \mathbf{D}|_{s-\frac34,b}\right).
	\end{equation}
	Therefore, $\mathbf{M}$ is a contraction map on $L_{s,b}$. By the Banach fixed-point theorem, there exists a unique solution to \eqref{eM1}--\eqref{eM2} in $L_{s,b}$. Furthermore, the continuous dependence of the solution on the initial data follows as a standard consequence of the contraction mapping principle. This completes the proof of Theorem~\ref{main thm}.

	\subsection{Local well-posedness for 1D} \label{proof1d}
	
	In one-dimensional case, we work in the null-coordinate framework introduced in Section~\ref{nf1d}. Thanks to the product-space characterization of the $X^{s,b}$ due to \eqref{X} under the null coordinate transformation, all derivative nonlinearities can be estimated through the product estimates of Keel and Tao \cite{KT}, i.e. Lemma~\ref{uvproduct}. Consequently, the proof becomes significantly more direct than in two dimensions. 
	
	Let $s>b>\frac12$, we first define the solution space
	\begin{equation}\label{Ls1}
		\begin{aligned}
			\widetilde{L}_{s, b}=
			& \{(\boldsymbol{\phi}, \mathbf{A}, N)
			\in X^{s,b}(\mathbb{R}^{1+1}) \times X^{s-1,b}(\mathbb{R}^{1+1}) \times X^{s,b}(\mathbb{R}^{1+1}) : 
			\\
			& \| \boldsymbol{\phi}\|_{X^{s,b}}+\|\mathbf{A}\|_{X^{s-1, b}}+\|N\|_{X^{s, b}} \leq 2\left(1+C_{0}\right) M_{0},
			\\
			& (\boldsymbol{\phi}, \partial_{t} \boldsymbol{\phi})|_{t=0}=(\boldsymbol{\phi}_{0}, \boldsymbol{\phi}_{1}),(\mathbf{A}, \partial_{t} \mathbf{A})|_{t=0}=(\mathbf{a}, \dot{\mathbf{a}}),(N, \partial_{t} N)|_{t=0}=(n_0, n_1)
			\},
		\end{aligned}
	\end{equation}
	where
	\[
	M_0=\left\|\boldsymbol{\phi}_0\right\|_{H^s}+\left\|\boldsymbol{\phi}_1\right\|_{H^{s-1}}+\|\mathbf{a}\|_{H^{s-1}}+\left\|\dot{\mathbf{a}}\right\|_{H^{s-2}}+\left\|n_0\right\|_{H^s}+\left\|n_1\right\|_{H^{s-1}}.
	\]
	Analogous to the construction in the 2D case, we define the Picard iteration map $\mathbf{M}$ via the identical time-localization procedure \eqref{dM}. By construction, for any $(\boldsymbol{\phi}, \mathbf{A}, N) \in \widetilde{L}_{s,b}$, the mapped variables strictly satisfy the following localized wave equations:
	\begin{equation}\label{eM11}
		\begin{cases}
			\square \mathbf{M} \boldsymbol{\phi}=\mathbf{I}, \quad(t, x) \in [0, T] \times \mathbb{R},
			\\
			\mathbf{M} \boldsymbol{\phi}|_{t=0}=\boldsymbol{\phi}_0, \quad \partial_t \mathbf{M} \boldsymbol{\phi}|_{t=0}=\boldsymbol{\phi}_1
		\end{cases}
	\end{equation}
	and
	\begin{equation}\label{eM12}
		\left\{
		\begin{array}{l}
		\square \mathbf{M} B_0=J_0, \quad (t, x) \in[0, T] \times \mathbb{R},
		\\
		\mathbf{M} B_0|_{t=0}=0, \;
		\partial_t\mathbf{M} B_0|_{t=0}=-a_1,
		\end{array}
		\right.
		\qquad
		\left\{
		\begin{array}{l}
		\square \mathbf{M}B_1 = J_1, \quad (t, x) \in[0, T] \times \mathbb{R},
		\\
		\mathbf{M}B_1|_{t=0}=0,\;
		\partial_t\mathbf{M}B_1|_{t=0}=a_0,
		\end{array}
		\right.
	\end{equation}
	and
	\begin{equation}\label{eM13}
		\begin{cases}
			\square \mathbf{M} N=J, 
			& (t, x) \in[0, T] \times \mathbb{R},
			\\
			\left.\mathbf{M} N\right|_{t=0}=n_0, 
			& \left.\partial_t \mathbf{M} N\right|_{t=0}=n_1,
		\end{cases}
	\end{equation}
	where $\mathbf{I}=(I_1,I_2,I_3), J_{\mu}, J$ represent the right side of \eqref{1wave1}--\eqref{1wave4}. 
	
	We first apply the linear estimates established in Lemma~\ref{energy1}, it follows that
	\begin{align}
		& \|\mathbf{M} \boldsymbol{\phi}\|_{X^{s,b}} \leq C\left(\left\|\boldsymbol{\phi}_{0}\right\|_{H^{s}}+\left\|\boldsymbol{\phi}_{1}\right\|_{H^{s-1}}+T^{\frac{\epsilon}{4}} \|\mathbf{I}\|_{X^{s-1, b+\epsilon-1}}\right) ,
		\label{Mphi}
		\\
		& \|\mathbf{M}B_{\mu}\|_{X^{s,b}} \leq C\left(\|\mathbf{a}\|_{H^s}+T^{\frac{\epsilon}{4}}\|J_{\mu}\|_{X^{s-1, b+\epsilon-1}}\right), \quad \mu=0,1,
		\label{MB}
		\\
		& \|\mathbf{M}N\|_{X^{s,b}} \leq C\left(\|n_0\|_{H^{s}}+\|n_1\|_{H^{s-1}}+T^{\frac{\epsilon}{4}}\|J\|_{X^{s-1, b+\epsilon-1}}\right).
		\label{MN}
	\end{align}
	Therefore, it remains to estimate the nonlinear terms $\|\mathbf{I}\|_{X^{s-1, b+\epsilon-1}}, \|J_{\mu}\|_{X^{s-1, b+\epsilon-1}}$ and $\|J\|_{X^{s-1, b+\epsilon-1}}$. Each nonlinear term can be written as a linear combination of the null forms $Q_0(X^{s,b}, X^{s,b})$ and $Q_1(X^{s,b}, X^{s,b})$. The null-coordinate formulation allows them to be estimated through the product estimates of Lemma~\ref{uvproduct}.
	
	Applying the product estimates established in Lemma~\ref{uvproduct} immediately yields
	\begin{equation}\label{Q0}
		\begin{aligned}
			\|Q_0(f,g)\|_{H_u^{s-1} H_v^{b-1}} 
			& \lesssim \|\partial_u f \partial_v g\|_{H_u^{s-1} H_v^{b-1}} + \|\partial_v f \partial_u g\|_{H_u^{s-1} H_v^{b-1}}
			\\
			& \lesssim \|\partial_u f\|_{H_u^{s-1} H_v^b} \|\partial_v g\|_{H_u^b H_v^{b-1}} + \|\partial_v f\|_{H_u^b H_v^{b-1}} \|\partial_u g\|_{H_u^{s-1} H_v^b}
			\\
			& \lesssim \|f\|_{H_u^s H_v^b} \|g\|_{H_u^b H_v^b} + \|f\|_{H_u^b H_v^b} \|g\|_{H_u^s H_v^b}.
		\end{aligned}
	\end{equation}
	Similarly
	\begin{equation}\label{Q1}
		\|Q_1(f,g)\|_{H_u^{s-1} H_v^{b-1}} \lesssim \|f\|_{H_u^s H_v^b} \|g\|_{H_u^b H_v^b} + \|f\|_{H_u^b H_v^b} \|g\|_{H_u^s H_v^b}.
	\end{equation}
	Combining \eqref{X}, \eqref{Q0}, \eqref{Q1} and the explicit expressions of the nonlinearities immediately gives
	\begin{align}
		\|\mathbf{I}\|_{X^{s-1, b+\epsilon-1}} 
		& \lesssim P \left( \|\boldsymbol{\phi}\|_{X^{s, b}}, \|B_{\mu}\|_{X^{s, b}}, \|N\|_{X^{s, b}} \right) \left( \|\boldsymbol{\phi}\|_{X^{s, b}}+\|B_{\mu}\|_{X^{s, b}}+\|N\|_{X^{s,b}} \right),
		\label{eI}
		\\
		\|J_{\mu}\|_{X^{s-1, b+\epsilon-1}}
		& \lesssim P \left( \|\boldsymbol{\phi}\|_{X^{s, b}}, \|N\|_{X^{s, b}} \right) \left( \|\boldsymbol{\phi}\|_{X^{s, b}}+\|N\|_{X^{s,b}} \right),
		\label{eJ01}
		\\
		\|J\|_{X^{s-1, b+\epsilon-1}}
		& \lesssim P \left( \|\boldsymbol{\phi}\|_{X^{s, b}}, \|B_{\mu}\|_{X^{s, b}}, \right) \left( \|\boldsymbol{\phi}\|_{X^{s, b}}+\|B_{\mu}\|_{X^{s, b}}\right),
		\label{eJ2}
	\end{align}
	where \(P\) is a polynomial satisfying \(P(0)=0\). Let  $(\boldsymbol{\phi},\mathbf{A},N), (\boldsymbol{\psi},\mathbf{D},K) \in \widetilde{L}_{s,b}$. By \eqref{Mphi}--\eqref{MN}, we obtain
	\[
	\begin{aligned}
		& \|\mathbf{M} \boldsymbol{\phi}-\mathbf{M} \boldsymbol{\psi}\|_{X^{s,b}} + \|\mathbf{M} \mathbf{A}-\mathbf{M} \mathbf{D}\|_{X^{s-1,b}} + \|\mathbf{M} N-\mathbf{M} K\|_{X^{s,b}}
		\\
		& \leq C T^{\frac{\epsilon}{4}} P \left( \|\boldsymbol{\phi}\|_{X^{s,b}},\|\boldsymbol{\psi}|_{X^{s,b}},\|\mathbf{A}\|_{X^{s-1,b}},\|\mathbf{D}\|_{X^{s-1,b}},\|N\|_{X^{s,b}},\|K\|_{X^{s,b}} \right) \cdot
		\\
		& \quad \left( \|\boldsymbol{\phi} - \boldsymbol{\psi}\|_{X^{s,b}} + \|\mathbf{A} - \mathbf{D}\|_{X^{s-1,b}} + \|N-K\|_{X^{s,b}} \right).
	\end{aligned}
	\]
	Then for sufficiently small $T$, this implies
	\begin{equation}\label{contraction1}
		\begin{aligned}
			& \|\mathbf{M} \boldsymbol{\phi} - \mathbf{M} \boldsymbol{\psi}\|_{X^{s,b}} + \|\mathbf{M} \mathbf{A} - \mathbf{M D}\|_{X^{s-1,b}} + \|\mathbf{M} N - \mathbf{M} K\|_{X^{s,b}} 
			\\
			& \leq \frac{1}{2} \left(\|\boldsymbol{\phi} - \boldsymbol{\psi}\|_{X^{s,b}} + \|\mathbf{A} - \mathbf{D}\|_{X^{s-1,b}} + \|N - K\|_{X^{s,b}} \right).
		\end{aligned}
	\end{equation}
	Therefore, $\mathbf{M}$ maps $\widetilde{L}_{s,b}$ into itself and is a contraction. By the Banach fixed-point theorem, there exists a unique solution to \eqref{eM11}--\eqref{eM13} in $\widetilde{L}_{s,b}$. Furthermore, the continuous dependence of the solution on the initial data follows as a standard consequence of the contraction mapping principle. This completes the proof of Theorem~\ref{main thm 1dim}. 
	

	\section{Appendix}\label{appendix}
	In this appendix, we first provide a detailed derivation of coupled wave system of the ($1+d$)-dimensional CS-$O(3)$ sigma model from the Euler–Lagrange equations, under the Lorenz gauge $\partial_{\mu} A^{\mu} = 0$. Next, we present a proof of Proposition~\ref{energy1}.
	
	\subsection{Wave system in 2D}
	The Euler–Lagrange equations under the Lorenz gauge take the form:
	\begin{align}
		D_{\mu} D^{\mu} \boldsymbol{\phi}+\boldsymbol{\phi}\left(D_{\mu} \boldsymbol{\phi} \cdot D_{\mu} \boldsymbol{\phi}\right) 
		& =-\frac{1}{\kappa^{2}}\left(\boldsymbol{\phi}\left(n \cdot \boldsymbol{\phi}\right)-n(\boldsymbol{\phi} \cdot \boldsymbol{\phi})\right)\left(1-n \cdot \boldsymbol{\phi}\right)^{2}\left(1+2 n \cdot \boldsymbol{\phi}\right),  
		\label{2EL11}
		\\
		\kappa F_{01} 
		& = \left\langle n \times \boldsymbol{\phi}, D_{2} \boldsymbol{\phi}\right\rangle, 
		\label{2EL22}
		\\
		\kappa F_{02} 
		& =  \left\langle n \times \boldsymbol{\phi}, D_{1} \boldsymbol{\phi}\right\rangle,  
		\label{2EL33}
		\\
		\partial_{\mu} A^{\mu} 
		& = 0,
		\label{2Lorenz}
	\end{align}
	supplemented by the constraint equation
	\begin{equation}
		\kappa F_{12} = - \left\langle n \times \phi, D_{0} \boldsymbol{\phi}\right\rangle
		\label{ELconstraint}
	\end{equation}
	and the initial data
	\begin{equation}
		A_{\mu}(0, \cdot)=a_{\mu}, \quad \partial_{t} A_{\mu}(0, \cdot)=\dot{a}_{\mu}, \quad \boldsymbol{\phi}(0, \cdot)=\boldsymbol{\phi}_{0}, \quad \partial_{t} \boldsymbol{\phi}(0, \cdot)=\boldsymbol{\phi}_{1}
		\label{initial}
	\end{equation}
	with $\langle \boldsymbol{\phi}_0, \boldsymbol{\phi}_1 \rangle = 0 $. 
	
	Next, we write \eqref{2EL11}-\eqref{ELconstraint} as wave equations. Notice that
	$$
	\begin{aligned}
		D_\mu D^\mu \phi 
		& =\partial_\mu D^\mu \phi+A_\mu\left(n \times D^\mu \phi\right) 
		\\
		& =\partial_\mu \partial^\mu \phi+\partial_\mu A^\mu(n \times \phi)+2 A^\mu n \times \partial_\mu \phi+A_\mu A^\mu n \times(n \times \phi)
	\end{aligned}
	$$
	and
	\[
	D_\mu \phi \cdot D^\mu \phi = \partial_\mu \phi \cdot \partial^\mu \phi+2 A_\mu \partial^\mu \phi \cdot n \times \phi+A_\mu A^\mu|n \times \phi|^2.
	\]
	Then we obtain
	\[
	\begin{aligned}
		\begin{split}
			\square \boldsymbol{\phi} = 
			& -\boldsymbol{\phi}\left(\partial_\mu \phi \cdot \partial^\mu \phi+2 A_{\mu} \partial^{\mu} \boldsymbol{\phi} \cdot\left(n \times \boldsymbol{\phi}\right)+A_{\mu} A^{\mu}\left|n \times \boldsymbol{\phi}\right|^{2}\right)-2 A^{\mu}\left(n \times \partial_{\mu} \boldsymbol{\phi}\right) 
			\\ 
			& -A_{\mu} A^{\mu} n \times\left(n \times \boldsymbol{\phi}\right)-\frac{1}{\kappa^{2}}\left(\boldsymbol{\phi}\left(n \cdot \boldsymbol{\phi}\right)-n(\boldsymbol{\phi} \cdot \boldsymbol{\phi})\right)\left(1-n \cdot \boldsymbol{\phi}\right)^{2}\left(1+2 n \cdot \boldsymbol{\phi}\right).
		\end{split}
	\end{aligned}
	\]
	
	For $F_{\mu \nu}$, we differentiate \eqref{ELconstraint} with respect to  $t$ and \eqref{2EL33} with respect to $x_{2}$ to obtain
	$$
	\begin{aligned}
		& \kappa\left(\partial_0 \partial_0 A_1-\partial_0 \partial_1 A_0\right)=\partial_0\left(D_2 \boldsymbol{\phi} \cdot n \times \boldsymbol{\phi}\right), 
		\\
		& \kappa\left(\partial_2 \partial_1 A_2-\partial_2 \partial_2 A_1\right)=-\partial_2\left(D_0 \boldsymbol{\phi} \cdot n \times \boldsymbol{\phi}\right) .
	\end{aligned}
	$$
	Summing these two equations and using the Lorenz gauge condition yields
	$$
	\begin{aligned}
		\kappa \square A_{1}
		&=\partial_{0}\left(D_{2} \boldsymbol{\phi} \cdot n \times \boldsymbol{\phi}\right)-\partial_{2}\left(D_{0} \boldsymbol{\phi} \cdot n \times \boldsymbol{\phi}\right)
		\\
		& = D_{0} D_{2} \boldsymbol{\phi} \cdot n \times \boldsymbol{\phi}+D_{2} \boldsymbol{\phi} \cdot n \times D_{0} \boldsymbol{\phi}-D_{2} D_{0} \boldsymbol{\phi} \cdot n \times \boldsymbol{\phi}-D_{0} \boldsymbol{\phi} \cdot n \times D_{2} \boldsymbol{\phi} 
		\\
		& = F_{02}|n \times \boldsymbol{\phi}|^{2}+2 n \cdot \partial_{0} \boldsymbol{\phi} \times \partial_{2} \boldsymbol{\phi}+A_{2} \partial_{0}|n \times \boldsymbol{\phi}|^{2}-A_{0} \partial_{2}|n \times \boldsymbol{\phi}|^{2}
		\\
		& =2 n \cdot \partial_{0} \boldsymbol{\phi} \times \partial_{2} \boldsymbol{\phi}+\partial_{0}\left(A_{2}|n \times \boldsymbol{\phi}|^{2}\right)-\partial_{2}\left(A_{0}|n \times \boldsymbol{\phi}|^{2}\right).
	\end{aligned}
	$$
	Similarly, for $\kappa \Box A_0, \kappa \Box A_2$ we have:
	\[
	\begin{aligned}
		\begin{split}
			\square A_{\mu} 
			& =\frac{1}{\kappa} \epsilon_{\mu \nu \rho}\left(\frac{1}{2} F^{\nu \rho}|n \times \boldsymbol{\phi}|^{2}+\partial^{\nu} \boldsymbol{\phi} \times \partial^{\rho} \boldsymbol{\phi} \cdot n+A^{\rho} \partial^{\nu}(\boldsymbol{\phi} \times(n \times \boldsymbol{\phi}) \cdot n)\right) 
			\\
			& =\frac{1}{\kappa} \epsilon_{\mu \nu \rho}\left(n \cdot\left(\partial^{\nu} \boldsymbol{\phi} \times \partial^{\rho} \boldsymbol{\phi}\right)+A^{\rho} \partial^{\nu}\left(\left(\boldsymbol{\phi} \times\left(n \times \boldsymbol{\phi}\right)\right) \cdot n\right)\right).
		\end{split}
	\end{aligned}
	\]

	\subsection{Wave system in 1D}
	The Chern--Simons gauged $O(3)$ sigma system in $\mathbb{R}^{1+1}$ can be regarded as a dimensional reduction of the model in $\mathbb{R}^{1+2}$, which is given by
	\begin{align}
		D_\mu D^\mu \boldsymbol{\phi} 
		& + \boldsymbol{\phi}(\langle D^\mu \boldsymbol{\phi}, D_\mu \boldsymbol{\phi} \rangle + \phi_3 U(\phi_3, N))  = U(\phi_3, N)n , 
		\label{11EL1}
		\\
		\kappa F_{01} 
		& = N|n \times \boldsymbol{\phi}|^2, 
		\label{11EL2}
		\\
		\kappa \partial_0 N 
		& = -\langle n \times \boldsymbol{\phi}, D_1 \boldsymbol{\phi} \rangle, 
		\label{11EL3}
		\\
		\kappa \partial_1 N 
		& = -\langle n \times \boldsymbol{\phi}, D_0 \boldsymbol{\phi} \rangle,
		\label{11EL4}
		\\
		\partial_{\mu} A^{\mu} 
		& = 0.
		\label{1Lorenz}
	\end{align}
	This syetem is invariant under gauge transformation \eqref{gaugetrans} together with $N \rightarrow N$, thus a solution to it is formed by a class of gauge-equivalent pairs $(\boldsymbol{\phi}, A_{\mu}, N)$.
	
	The reformulation of \eqref{11EL1}--\eqref{1Lorenz} is similar to 2D. For \eqref{11EL1}, we have
	\[
	\begin{split}
		\Box \boldsymbol{\phi} 
		& = - \boldsymbol{\phi} \left( \partial_{\mu} \boldsymbol{\phi} \cdot \partial^{\mu} \boldsymbol{\phi} + 2 A_{\mu} \partial^{\mu} \boldsymbol{\phi} \cdot (n \times \boldsymbol{\phi}) + A_{\mu} A^{\mu} |n \times \boldsymbol{\phi}|^2 + \phi_3 U(\phi_3, N)  \right)
		\\
		& - 2 A_{\mu} \partial^{\mu}(n \times \boldsymbol{\phi}) - A_{\mu} A^{\mu} n \times (n \times \boldsymbol{\phi}).
	\end{split}
	\]
	Then applying $\partial_0$ and $\partial_1$ to \eqref{11EL2} and Lorenz gauge condition, we obtain
	\[
	\begin{aligned}
		\kappa \Box A_0
		& = \partial_1(N|n \times \boldsymbol{\phi}|^2),
		\\
		\kappa \Box A_1
		& = \partial_0(N|n \times \boldsymbol{\phi}|^2),
	\end{aligned}
	\]
	Applying $\partial_0$ to \eqref{11EL3}, $\partial_0$ to \eqref{11EL4} and subtracting the resulting identities, we obtain
	\[
	\kappa \Box N = D_1(n \times \boldsymbol{\phi}) \cdot D_0 \boldsymbol{\phi} - D_0(n \times \boldsymbol{\phi}) \cdot D_1 \boldsymbol{\phi} - (\phi_1^2 + \phi_2^2)F_{01},
	\]
	due to the identity
	\[
	D_{\mu} D_{\nu} \boldsymbol{\phi} - D_{\nu} D_{\mu} \boldsymbol{\phi} = F_{\mu \nu} (n \times \boldsymbol{\phi}).
	\]
	
	\subsection{Proof of Proposition~\ref{energy1}} \label{proof 3.1}
	 We present a proof of Proposition~\ref{energy1}, inspired by Selberg’s paper \cite{Selberg}. Let us restate the Proposition in a more precise form.
	\begin{proposition} \label{energy1dim}
		Assume $s > \frac12, b \in\left(\frac{1}{2}, 1\right), \epsilon \in[0,1-b]$. Consider the Cauchy problem for the linear wave equation
		\begin{equation}\label{1linearw}
			\begin{cases}
				\square w=F(t,x), \quad(t, x) \in \mathbb{R}^{1+d}, 
				\\
				\left.w\right|_{t=0}=f,\left.\quad \partial_t w\right|_{t=0}=g,
			\end{cases}
		\end{equation}
		let $f, g$ and $F$ satisfy $f \in H^{s}, g \in H^{s-1}$, and $F \in X^{s-1, b+\epsilon-1}$.
		
		Let $0<T<1$ and define
		\begin{equation}\label{1defw}
			w(t)=\theta(t)w_0 + \theta_T(w_1+ w_2),
		\end{equation}
		where
		\begin{equation}\label{1defw012}
		\begin{aligned}
			& w_0=\cos (t D) f+D^{-1} \sin (t D) g ,
			\\
			& F_{1}=\eta\left(T^{\frac14} \Lambda_{-}\right) F, \quad F_{2}=\left(1-\eta\left(T^{\frac14} \Lambda_{-}\right)\right) F,
			\\
			& w_1= - \int_{0}^{t} D^{-1} \sin \left(\left(t-t^{\prime}\right) D\right) F_{1}\left(t^{\prime}\right) d t^{\prime} , 
			\\
			& w_2=\square^{-1} F_2,
		\end{aligned}
		\end{equation}
		Then, the function $w$ defined in \eqref{1defw}-\eqref{1defw012} is the unique solution to \eqref{1linearw} on $[0, T] \times \mathbb{R}^d$ such that $w \in C([0, T] ; H^{s}) \cap C^1([0, T] ; H^{s-1})$ and satisfies the following estimate:
		\begin{equation}\label{energy11}
			\|w\|_{X^{s, b}} \leq C_{0}(\|f\|_{H^{s}}+\|g\|_{H^{s-1}}+T^{\frac{\epsilon}{4}}\|F\|_{X^{s-1, b+\epsilon-1}}),
		\end{equation}
		where $C_{0}$ only depends on $\theta$ and $b$.
	\end{proposition}

	We first establish the corresponding estimates for $\theta(t) w_0, \theta(t) w_1, \theta(t) w_2$ sequentially.
	
	{\bfseries Step 1: the estimates for $\theta(t) w_0 = \theta(t) \cos (t D) f + D^{-1} \sin (t D) g$.}
	\begin{proposition}\label{w0}
		Let \( s > \frac12, b > \frac12, \theta \in C_c^\infty(\mathbb{R}) \) and \( (f,g) \in H^s \times H^{s-1} \), then
		\begin{align}
			\|\theta(t)e^{\pm \operatorname{i}tD}f\|_{X^{s,b}} 
			&\lesssim \|\theta\|_{H^{s+b}} \|f\|_{H^s},
			\label{f}
			\\
			\|\theta(t)\cos(tD) f\|_{X^{s,b}} 
			&\lesssim \|\theta\|_{H^{s+b}} \|f\|_{H^s}.
			\label{w0f}
		\end{align}
		If \( |r| \leq 1 \), we have
		\begin{equation}\label{g}
			\|\theta(t)e^{\operatorname{i} r tD}g\|_{X^{s,b}} \lesssim \|\theta\|_{H^{s+b}} \|g\|_{H^{s}}.
		\end{equation}
		Moreover, we get
		\begin{equation}\label{w0g}
			\|\theta(t)D^{-1}\sin(tD) g\|_{X^{s,b}} \lesssim (\|\theta\|_{H^{s+b}} + \|t\theta\|_{H^{s+b}}) \|g\|_{H^{s-1}}.
		\end{equation}
	\end{proposition}
	
	\begin{proof}
		The Fourier transform of \(\theta(t)e^{\pm itD}f\) is \(\widehat{\theta}(\tau \mp |\xi|)\widehat{f}(\xi)\), and
		\begin{align*}
			\|\theta(t)e^{\pm \operatorname{i}tD}f\|_{X^{s,b}}^2
			& = \int_{\mathbb{R}^{1+d}} (1 + |\tau| + |\xi|)^{2s} (1 + ||\tau| - |\xi|| )^{2b} |\widehat{\theta}(\tau \mp |\xi|)\widehat{f}(\xi)|^2  d\tau d\xi 
			\\
			& \overset{\sigma = \tau \mp |\xi|}{\leq} \int_{\mathbb{R}^{1+d}} (1+|\sigma \pm |\xi||+|\xi|)^{2s} (1+|\sigma|)^{2b} |\widehat{\theta}(\sigma)\widehat{f}(\xi)|^2 d\tau d\xi 
			\\
			& \leq \int_{\mathbb{R}^{1+d}} (1+2|\xi|)^{2s} (1+|\sigma|)^{2s+2b} |\widehat{\theta}(\sigma)\widehat{f}(\xi)|^2 d\tau d\xi 
			\\
			& \lesssim \|\theta\|_{H^{s+b}} \|f\|_{H^s}.
		\end{align*}
		This proves \eqref{f}, which in turn implies \eqref{w0f} by using the identity \( \cos(tD)f = \frac{1}{2}(e^{\operatorname{i}tD}f + e^{-\operatorname{i}tD}f) \).
		
		For \eqref{g}, since the sapce-time Fourier transform of $\theta(t)e^{\operatorname{i} r tD}g$ equals $\widehat{\theta}(\tau - r |\xi|)\widehat{g}(\xi)$, and
		\begin{align*}
			\|\theta(t)e^{\operatorname{i}\rho tD}g\|_{X^{s,b}}^2
			& = \int_{\mathbb{R}^{1+d}} (1 + |\tau| + |\xi|)^{2s} (1 + \big||\tau| - |\xi|\big| )^{2b} |\widehat{\theta}(\tau - r |\xi|)\widehat{g}(\xi)|^2  d\tau d\xi 
			\\
			& \overset{\sigma =  \tau - r |\xi|}{=} \int_{\mathbb{R}^{1+d}} (1+\big|\sigma + r |\xi|\big|+|\xi|)^{2s} (1+\big||\sigma + r |\xi||- |\xi|\big|)^{2b} |\widehat{\theta}(\sigma)\widehat{g}(\xi)|^2 d\tau d\xi
			\\
			& \leq \int_{\mathbb{R}^{1+d}} (1+|\sigma|+2|\xi|)^{2s} (1+|\sigma|)^{2b} |\widehat{\theta}(\sigma)\widehat{g}(\xi)|^2 d\tau d\xi
			\\
			& \lesssim \|\theta\|_{H^{s+b}} \|g\|_{H^{s}}.
		\end{align*}
		This gives \eqref{g}. For the estimate of another part of homogeneous solution to equation \eqref{w0g}, we give a decomposition $g = g_1 + g_2$ where
		\[
		\operatorname{supp} \widehat{g}_1 \subseteq \{ \xi : |\xi| < 1 \}, \quad \operatorname{supp} \widehat{g}_2 \subseteq \{ \xi : |\xi| \geq 1 \}.
		\]
		For $g_1$, since
		\[
		\theta(t)D^{-1}\sin(tD)g_1 = \int_0^1 t \theta(t) e^{\operatorname{i}(2\rho-1)tD}g_1 d\rho.
		\]
		Due to \eqref{g} and $\operatorname{supp} \widehat{g}_1 \subseteq \{ \xi : |\xi| < 1 \}$, we obtain
		\begin{equation} \label{g1}
			\|\theta(t)D^{-1}\sin(tD)g_1\|_{{X^{s,b}}} \lesssim \|t\theta(t)\|_{H^{s+b}} \|g_1\|_{H^{s-1}}.
		\end{equation}
		For $g_2$, we have
		\begin{equation} \label{g2}
			\begin{aligned}
				\|\theta(t)D^{-1}\sin(tD)g_2\|_{{X^{s,b}}} 
				& = \|\theta(t) \sin(tD)(D^{-1}g_2)\|_{{X^{s,b}}} 
				\\
				& \lesssim \|\theta(t)\|_{H^{s+b}} \|D^{-1}g_2\|_{H^s} 
				\\
				&\lesssim \|\theta(t)\|_{H^{s+b}} \|g_2\|_{H^{s-1}}.
			\end{aligned}
		\end{equation}
		Combining \eqref{g1} with \eqref{g2} shows that \eqref{w0g} holds. Therefore, the proof of Proposition~\ref{w0} is complete.
	\end{proof}

	{\bfseries Step 2: the estimates for} $w_2 = \Box^{-1} F_2 $, where $ F_2 = (1 - \eta(\Lambda_-))F $. We define
	\[
	\mathcal{N} = \{(\tau, \xi) \in \mathbb{R}^{1+d} : ||\tau| - |\xi|| < 1\}.
	\]
	Then  \( \operatorname{supp} \widetilde{F_2} \subseteq \mathbb{R}^{1+d} \setminus \mathcal{N} \), we have
	\begin{proposition} \label{w2}
		Assume $s > \frac12, b > \frac12$, then the following estimate holds:
		\begin{equation}\label{ew2}
			\|w_2\|_{X^{s,b}} \lesssim \| F_2 \|_{X^{s-1,b-1}}.
		\end{equation}
	\end{proposition}

	\begin{proof}
		Note that 
		\[
		\widetilde{w_2}(\tau, \xi) = (|\tau|^2 - |\xi|^2)^{-1} \widetilde{F_2}(\tau, \xi).
		\]
		Consequently, we can compute
		\begin{align*}
			\|w_2\|_{X^{s,b}}
			&\lesssim \| \langle |\tau| + |\xi| \rangle^s \langle |\tau| - |\xi| \rangle^b (|\tau| - |\xi|)^{-1} (|\tau| + |\xi|)^{-1} \widetilde{F_2}(\tau, \xi) \|_{L^2(\mathbb{R}^{1+d})} 
			\\
			&\lesssim \| \langle |\tau| + |\xi| \rangle^{s-1} \langle |\tau| - |\xi| \rangle^{b-1} \widetilde{F_2}(\tau, \xi) \|_{L^2(\mathbb{R}^{1+d})} 
			= \|F_2\|_{X^{s-1,b-1}}.
		\end{align*}
		Therefore, the proof of this proposition is complete.
	\end{proof}
	
	{\bfseries Step 3: the estimates for $\theta(t)w_1$.} We recall
	\[
	w_1= - \int_{0}^{t} D^{-1} \sin \left(\left(t-t^{\prime}\right) D\right) F_{1}\left(t^{\prime}\right) d t^{\prime} ,
	\]
	where $F_1 = \eta(\Lambda_{-}) F$, then \( \operatorname{supp} \widetilde{F_1} \subseteq 4 \sqrt{2} \mathcal{N} \). To bound $w_1$, let us first introduce the following decomposition.
	
	\begin{lemma}[\cite{Selberg}, Proposition 16]\label{w1}
		Let \( s \in \mathbb{R} \), \( b \in (\frac{1}{2}, 1) \) and \( c \) be a positive constant with \( c \geq 2 \). Suppose that
		\[
		2 + \bigl||\tau| - |\xi|\bigr| \leq c, \quad \text{for } (\tau, \xi) \in \operatorname{supp} \widetilde{F_1}. 
		\]
		Then there exist \( f_j^\pm \in H^s \), \( g_j \in C([0,1], H^{s-1}) \) for \( j \geq 1 \) such that
		\begin{gather*}
			\operatorname{supp} \widehat{f_j^\pm} \subseteq \{ \xi : |\xi| \geq c \}, 
			\\
			\operatorname{supp} \widehat{g_j(\rho)} \subseteq \{ \xi : |\xi| < c \}, 
			\\
			\| f_j^\pm \|_{H^s}, \sup_{0 < \rho < 1} \| g_j(\rho) \|_{H^{s-1}} \lesssim c^{j-\frac{1}{2}} \| F_1 \|_{s-1,0}.
		\end{gather*}
		and
		\begin{equation}\label{dew1}
		w_1(t) = \sum_{j=1}^\infty \frac{t^{j+1}}{j!} \int_0^1 \mathrm{e}^{\operatorname{i}t(2\rho-1)D} g_j(\rho) d\rho
		+ \sum_{j=1}^\infty \frac{t^j}{j!} ( \mathrm{e}^{\operatorname{i}tD} f_j^+ + \mathrm{e}^{-\operatorname{i}tD} f_j^- )
		+ R_+(t) + R_-(t).
		\end{equation}
		\( R_+(t) \) and \( R_-(t) \) are given by
		\[
		\begin{aligned}
			\widehat{R_+(t)} (\xi) 
			& = -\frac{1}{4\pi |\xi|} \int_{-\infty}^0 \frac{\mathrm{e}^{\operatorname{i}t\tau} - \mathrm{e}^{\operatorname{i}t|\xi|}}{|\tau| + |\xi|} \widetilde{F_{1,2}}(\tau, \xi) d\tau, 
			\\
			\widehat{R_-(t)} (\xi)
			& = -\frac{1}{4\pi |\xi|} \int_0^\infty \frac{\mathrm{e}^{\operatorname{i}t\tau} - \mathrm{e}^{-\operatorname{i}t|\xi|}}{|\tau| + |\xi|} \widetilde{F_{1,2}}(\tau, \xi) d\tau,
		\end{aligned}
		\]
		where \( \operatorname{supp} \widehat{F_{1,2}} \subseteq \{ \xi : |\xi| \geq c \} \).
	\end{lemma}

	\begin{proof}
		See Section 3.6.3 in \cite{Selberg}.
	\end{proof}
	
	Based on this decomposition, we prove the estimate for $\theta(t) w_1$.
	
	\begin{proposition}\label{ew1}
		Assume $s > \frac12, b \in (\frac{1}{2}, 1)$, and
		\begin{equation}\label{c}
			2 + \bigl||\tau| - |\xi|\bigr| \leq c, \quad \text{for } (\tau, \xi) \in \operatorname{supp} \widetilde{F_1},
		\end{equation} 
		then
		\begin{equation}\label{thetaw1}
			\|\theta(t) w_1\|_{X^{s,b}} \leq C \|F_1\|_{X^{s-1,b-1}},
		\end{equation}
		where $w_1$ is given by \eqref{dew1} and 
		\[
		\begin{aligned}
			C
			& =  \sum_{j=1}^{\infty} \frac{ c^{j + \frac32 -s -b} \| t^{j} \theta \|_{H^{s+b}}}{j!} +  \sum_{j=1}^{\infty} \frac{ c^{j + \frac52 -s -b} \| t^{j+1} \theta \|_{H^{s+b}}}{j!} 
			\\
			& + c^{\frac52 -s-b} ( \| t \theta \|_{L^2} + \|\theta\|_{\dot{H}^{b-1}} ) + c^{\frac52-2s-b} \|t \theta\|_{H^{s+b}}.
		\end{aligned}
		\]
	\end{proposition}
	
	\begin{proof}
		Lemma~\ref{w1} gives a decomposition $w_1 = w_{1,1} + w_{1,2}$ where $\operatorname{supp} w_{1,1} \subseteq \{\xi: |\xi| < c \} $ and $\operatorname{supp} w_{1,2} \subseteq \{\xi: |\xi| \geq c \} $.
		
		For low frequency part, 
		\[
		w_{1,1}(t) = \sum_{j=1}^{\infty} \frac{t^{j+1}}{j!} \int_{0}^{1} e^{\operatorname{i}t(2\rho-1)D} g_j(\rho) d\rho,
		\]
		by \eqref{g} and $\| g_j(\rho) \|_{H^{s-1}} \lesssim c^{j-\frac{1}{2}} \| F_1 \|_{s-1,0}$, we have
		\[
		\begin{aligned}
			\| \theta w_{1,1} \|_{X^{s,b}}
			& \leq \int_{0}^{1} \left\| \sum_{j=1}^{\infty} \frac{\theta(t) t^{j+1}}{j!} e^{\operatorname{i}t(2\rho-1)D} g_j(\rho) \right\|_{X^{s,b}} d\rho
			\\
			& \lesssim \left( \sum_{j=1}^{\infty} \frac{\| t^{j+1} \theta \|_{H^{s+b}}}{j!} \right) \sup_{0 \leq \rho \leq 1} \| g_j(\rho) \|_{H^s} 
			\\
			& \lesssim \left( \sum_{j=1}^{\infty} \frac{\| t^{j+1} \theta \|_{H^{s+b}}}{j!} \right) \cdot c \cdot \sup_{0 \leq \rho \leq 1} \| g_j(\rho) \|_{H^{s-1}}
			\\
			& \lesssim \left( \sum_{j=1}^{\infty} \frac{ c^{j + \frac12} \| t^{j+1} \theta \|_{H^{s+b}}}{j!} \right) \|F_1\|_{s-1,0}
		\end{aligned}
		\]
		Since \( \operatorname{supp} \widetilde{F_1} \subseteq 4 \sqrt{2} \mathcal{N} \), then $\bigl| |\tau| - |\xi| \bigr| \lesssim c$, this means $|\tau| \sim |\xi|$ and $\langle |\tau| + |\xi| \rangle \lesssim c \langle \xi \rangle$. For $s$ around $\frac12 +$, we have
		\[
		\langle \xi \rangle ^{s-1} \lesssim c^{1-s} \langle |\tau| + |\xi| \rangle ^{s-1}.
		\]
		Thus we get
		\begin{equation}\label{ew11}
			\begin{aligned}
				\| \theta w_{1,1} \|_{X^{s,b}}
				& \lesssim \left( \sum_{j=1}^{\infty} \frac{ c^{j + \frac12} \| t^{j+1} \theta \|_{H^{s+b}}}{j!} \right) \|F_1\|_{s-1,0}
				\\
				& \lesssim \left( \sum_{j=1}^{\infty} \frac{ c^{j + \frac52 -s -b} \| t^{j+1} \theta \|_{H^{s+b}}}{j!} \right) \|F_1\|_{X^{s-1,b-1}}.
			\end{aligned}
		\end{equation}
		
		For high frequency part,
		\[
		w_{1,2}(t) = \sum_{j=1}^{\infty} \frac{t^j}{j!} \left( \mathrm{e}^{\operatorname{i}tD} f_j^+ + \mathrm{e}^{-\operatorname{i}tD} f_j^- \right) + R_+(t) + R_-(t).
		\]
		First, by \eqref{f} and $\| f_j^\pm \|_{H^s} \lesssim c^{j-\frac{1}{2}} \| F_1 \|_{s-1,0}$ we have
		\begin{equation}\label{ew12}
			\begin{aligned}
				\left\| \theta(t) \sum_{j=1}^{\infty} \frac{t^j}{j!} \left( \mathrm{e}^{\operatorname{i}tD} f_j^+ + \mathrm{e}^{-\operatorname{i}tD} f_j^- \right) \right\|_{X^{s,b}} 
				& \lesssim \left( \sum_{j=1}^{\infty} \frac{ c^{j - \frac12} \| t^{j} \theta \|_{H^{s+b}}}{j!} \right) \|F_1\|_{s-1,0}
				\\
				& \lesssim \left( \sum_{j=1}^{\infty} \frac{ c^{j + \frac32 -s -b} \| t^{j} \theta \|_{H^{s+b}}}{j!} \right) \|F_1\|_{X^{s-1,b-1}}.
			\end{aligned}
		\end{equation}
		Next, since
		\[
		\widehat{\theta R_+}(\tau, \xi) = -\frac{1}{4\pi |\xi|} \int_{-\infty}^0 \frac{\widehat{\theta}(\tau - \lambda) - \widehat{\theta}(\tau - |\xi|)}{|\lambda| + |\xi|} \widehat{F_{1,2}}(\lambda, \xi) d\lambda,
		\]
		where \( \operatorname{supp} \widehat{F_{1,2}} \subseteq \{ \xi : |\xi| \geq c \} \) and \( \bigl| |\lambda| - |\xi| \bigr| \lesssim c\), it follows from Minkowski's inequality that
		\[
		\|\theta R_+ \|_{X^{s, b}} \lesssim \int_{-\infty}^0 \left\| B(\lambda, \xi) |\xi|^{-1} \widehat{F_{1,2}}(\lambda, \xi) \right\|_{L_\xi^2} d\lambda, 
		\]
		where
		\[
		B = \left\| (1 + |\tau| + |\xi|)^s (1 + \bigl| |\tau| - |\xi| \bigr|)^b \frac{\widehat{\theta}(\tau - \lambda) - \widehat{\theta}(\tau - |\xi|)}{|\lambda| + |\xi|} \right\|_{L^2_\tau}.
		\]
		Note that
		\[
		\frac{\widehat{\theta}(\tau - \lambda) - \widehat{\theta}(\tau - |\xi|)}{|\lambda| + |\xi|}
		= \int_{0}^{1} \widehat{\theta}^{\prime}(\tau - |\xi| + \rho(|\lambda| + |\xi|)) d\rho.
		\]
		To estimates $B$, we split the time and space frequency into two parts : $B_1$ and $B_2$. We define
		\[
		I = I(\lambda, \xi) = \left\{ \tau \in \mathbb{R} : \left| \tau - |\xi| \right| < 2(|\lambda| + |\xi|) \right\}.
		\]
		
		On $I$, since $|\lambda| \sim |\xi|$, we have $|\tau| \lesssim |\xi|$. Thus 
		\[
		\begin{aligned}
			B_1(\lambda,\xi) 
			& \lesssim \langle \xi \rangle ^s \left\|(1 + \bigl| |\tau| - |\xi| \bigr|)^b \frac{\widehat{\theta}(\tau - \lambda) - \widehat{\theta}(\tau - |\xi|)}{|\lambda| + |\xi|} \right\|_{L^2_\tau(I)}
			\\
			& \lesssim \langle \xi \rangle ^s \left\|\frac{\widehat{\theta}(\tau - \lambda) - \widehat{\theta}(\tau - |\xi|)}{|\lambda| + |\xi|} \right\|_{L^2_\tau(I)} 
			+ \langle \xi \rangle ^s \left\|\bigl| |\tau| - |\xi| \bigr|^b \frac{\widehat{\theta}(\tau - \lambda) - \widehat{\theta}(\tau - |\xi|)}{|\lambda| + |\xi|} \right\|_{L^2_\tau(I)}
			\\
			& \lesssim \int_{0}^{1} \langle \xi \rangle ^s \left\|\widehat{\theta}^{\prime}(\tau - |\xi| + \rho(|\lambda| + |\xi|)) \right\|_{L^2_\tau(I)} d\rho 
			+ \langle \xi \rangle ^s \left\| \bigl| |\lambda| + |\xi| \bigr|^{b-1} \widehat{\theta}(\tau - \lambda) - \widehat{\theta}(\tau - |\xi|) \right\|_{L^2_\tau(I)}.
		\end{aligned}
		\]
		Due to
		\[
		\big| \tau - |\xi| \big| < 2(|\lambda| + |\xi|), \quad |\tau - \lambda | \leq \big|\tau - |\xi|\big| + \big||\xi| - \lambda \big| \lesssim |\lambda| + |\xi|, \quad \text{for} \quad \tau \in I,
		\]
		and the fact that $b-1 < 0$, we have
		\[
		\begin{aligned}
			B_1(\lambda,\xi) 
			& \lesssim \langle \xi \rangle ^s \left(  \| t \theta \|_{L^2} + \left\| |\tau-\lambda|^{b-1} \widehat{\theta}(\tau - \lambda) \right\|_{L^2_\tau} + \left\| \big|\tau-|\xi|\big|^{b-1} \widehat{\theta}(\tau - |\xi|) \right\|_{L^2_\tau}  \right) 
			\\
			& \lesssim \langle \xi \rangle ^s \left( \| t \theta \|_{L^2} + \|\theta\|_{\dot{H}^{b-1}} \right).
		\end{aligned}
		\]
		
		On $\mathbb{R} \setminus I $, since $\bigl| |\tau| \pm |\xi| \bigr| \lesssim \left| \tau - |\xi| + \rho(|\lambda|+|\xi|) \right| $, we have
		\[
		\begin{aligned}
			B_2(\lambda,\xi) 
			& = \int_{0}^{1} \left\| \langle |\tau|+|\xi| \rangle^s \langle |\tau|-|\xi| \rangle^b \widehat{\theta}^{\prime}(\tau - |\xi| + \rho(|\lambda| + |\xi|)) \right\|_{L^2_\tau(\mathbb{R} \setminus I)} d\rho 
			\\
			& \lesssim \int_{0}^{1} \left\| \langle \tau - |\xi| + \rho(|\lambda| + |\xi|) \rangle^{s+b} \widehat{\theta}^{\prime}(\tau - |\xi| + \rho(|\lambda| + |\xi|)) \right\|_{L^2_\tau(\mathbb{R} \setminus I)} d\rho 
			\\
			& \lesssim \|t \theta\|_{H^{s+b}}.
		\end{aligned}
		\]
		
		In both cases we conclude, by applying Minkowski's inequality and with the fact that $|\xi| \geq c$, we obtain
		\begin{equation}\label{R+}
		\begin{aligned}
		\|\theta R_+ \|_{X^{s, b}}
		& \lesssim \int_{-\infty}^0 \left\| \left( B_1(\lambda, \xi) + B_2(\lambda, \xi) \right) |\xi|^{-1} \widehat{F_{1,2}}(\lambda, \xi) \right\|_{L_\xi^2} d\lambda,
		\\
		& \lesssim \left( \| t \theta \|_{L^2} + \|\theta\|_{\dot{H}^{b-1}} \right) \int_{-\infty}^0 \left\| \langle \xi \rangle ^s  |\xi|^{-1} \widehat{F_{1,2}}(\lambda, \xi) \right\|_{L_\xi^2} d\lambda + \|t \theta\|_{H^{s+b}} \int_{-\infty}^0 \left\||\xi|^{-1} \widehat{F_{1,2}}(\lambda, \xi) \right\|_{L_\xi^2} d\lambda
		\\
		& \lesssim c^{\frac12} \left( \| t \theta \|_{L^2} + \|\theta\|_{\dot{H}^{b-1}} \right) \|F_{1,2}\|_{s-1,0} + c^{\frac12-s} \|t \theta\|_{H^{s+b}} \|F_{1,2}\|_{s-1,0}
		\\
		& \lesssim \left( c^{\frac52 -s-b} ( \| t \theta \|_{L^2} + \|\theta\|_{\dot{H}^{b-1}} ) + c^{\frac52-2s-b} \|t \theta\|_{H^{s+b}} \right) \|F_1\|_{X^{s-1,b-1}}.
		\end{aligned}
		\end{equation}
		By a similar argument, we can get the estimate for $\theta(t) R_-$, we have
		\begin{equation}\label{R}
			\|\theta R_{\pm} \|_{X^{s, b}} \lesssim \left( c^{\frac52 -s-b} ( \| t \theta \|_{L^2} + \|\theta\|_{\dot{H}^{b-1}} ) + c^{\frac52-2s-b} \|t \theta\|_{H^{s+b}} \right) \|F_1\|_{X^{s-1,b-1}}.
		\end{equation}
		
		In conclusion, it follows form \eqref{ew11}, \eqref{ew12} and \eqref{R} that
		\[
		\|\theta(t) w_1\|_{X^{s,b}} \leq C \|F_1\|_{X^{s-1,b-1}},
		\]
		where 
		\[
		\begin{aligned}
			C
			& =  \sum_{j=1}^{\infty} \frac{ c^{j + \frac32 -s -b} \| t^{j} \theta \|_{H^{s+b}}}{j!} +  \sum_{j=1}^{\infty} \frac{ c^{j + \frac52 -s -b} \| t^{j+1} \theta \|_{H^{s+b}}}{j!} 
			\\
			& + c^{\frac52 -s-b} ( \| t \theta \|_{L^2} + \|\theta\|_{\dot{H}^{b-1}} ) + c^{\frac52-2s-b} \|t \theta\|_{H^{s+b}}.
		\end{aligned}
		\]
		Thus, the proof of Proposition~\ref{ew1} is complete.
	\end{proof}
	
	{\bfseries Proof of Proposition~\ref{energy1dim}.} By Proposition~\ref{w0},
	\begin{equation} \label{theta w0}
		\| \theta(t) w_0 \|_{X^{s,b}} \leq C (\|f\|_{H^s} + \|g\|_{H^{s-1}}), 
	\end{equation}
	where $C$ only depends on $\theta$ and $b$.
	
	Note that
	\[
	\operatorname{supp} \widetilde{F_1} \subseteq 4\sqrt{2} \, \mathcal{N}, \quad
	\operatorname{supp} \widetilde{F_2} \subseteq \mathbb{R}^{1+d} \setminus \mathcal{N}.
	\]
	and $\theta_T(t) = \theta(t/T)$, we can check
	\[
	\bigl| |\tau| - |\xi| \bigr| \leq 2\sqrt{2} T^{-\frac14}, \quad \text{for } (\tau, \xi) \in \operatorname{supp} \widetilde{F_1}.
	\]
	whence \eqref{c} holds when $c = 2 + 2\sqrt{2} T^{-\frac14}$. Thus, by Proposition~\ref{ew1},
	\begin{equation} \label{theta w1}
		\| \theta_T(t) w_1 \|_{X^{s,b}} \leq C_T \|F_1\|_{X^{s-1,b-1}} \leq C_T \|F_1\|_{X^{s-1,b+\epsilon-1}},
	\end{equation}
	where
	\[
	\begin{aligned}
		C_T
		& =  \sum_{j=1}^{\infty} \frac{ c^{j + \frac32 -s -b} \| t^{j} \theta_T \|_{H^{s+b}}}{j!} +  \sum_{j=1}^{\infty} \frac{ c^{j + \frac52 -s -b} \| t^{j+1} \theta_T \|_{H^{s+b}}}{j!} 
		\\
		& + c^{\frac52 -s-b} ( \| t \theta_T \|_{L^2} + \|\theta_T\|_{\dot{H}^{b-1}} ) + c^{\frac52-2s-b} \|t \theta_T\|_{H^{s+b}}.
	\end{aligned}
	\]
	Since
	\[
	\|t^k \theta_T\|_{H^{s+b}} \lesssim T^{k+1/2-s-b} \|\theta\|_{H^{s+b}} \quad \text{for } 0 < T \leq 1,
	\]
	and
	\[
	\|\theta_T\|_{\dot{H}^{b-1}} = T^{3/2 - b} \|\theta\|_{\dot{H}^{b-1}} \quad \text{for } \theta > 1/2,
	\]
	we get 
	\[
	\begin{aligned}
		C_T
		& =  \sum_{j=1}^{\infty} \frac{ (cT)^{j + \frac12 -s -b} \cdot c \cdot \| t^{j} \theta \|_{H^{s+b}}}{j!} +  \sum_{j=1}^{\infty} \frac{ (cT)^{j + \frac32 -s -b} \cdot c \cdot \| t^{j+1} \theta \|_{H^{s+b}}}{j!} 
		\\
		& + (cT)^{\frac32} \cdot c^{1-s-b} \| t \theta \|_{L^2} + (cT)^{\frac32-b} \cdot c^{1-s} \|\theta\|_{\dot{H}^{b-1}} + (cT)^{\frac32-s-b} \cdot c^{1-s} \|t \theta\|_{H^{s+b}}
		\\
		& \lesssim  \sum_{j=1}^{\infty} \frac{ T^{\frac34(j-s-b) + \frac18} \|t^{j} \theta \|_{H^{s+b}}}{j!} + \sum_{j=1}^{\infty} \frac{ T^{\frac34(j-s-b) + \frac78} \| t^{j+1} \theta \|_{H^{s+b}}}{j!} 
		\\
		& + T^{\frac14(s+b) + \frac78} \| t \theta \|_{L^2} + T^{\frac78 - \frac34 b + \frac14 s} \|\theta\|_{\dot{H}^{b-1}} + T^{\frac78 - \frac12 s - \frac34 b} \|t \theta\|_{H^{s+b}}.
	\end{aligned}
	\]
	Thus, since \( c \lesssim T^{-1/4} \) and \( \theta < 1 \), we conclude that \( C_T \leq C_\chi T^{1/8} \), where
	\[
	C_\chi \simeq \| \theta \|_{\dot{H}^{b-1}} + \| t \theta \|_{H^{s+b}} + \sum_{j=1}^\infty \frac{1}{j!} \left( \| t^{j+1} \theta \|_{H^{s+b}}  + \| t^j \theta \|_{H^{s+b}} \right).
	\]
	
	Next, since it is readily verified that
	\[
	\bigl| |\tau| - |\xi| \bigr| > \frac{1}{T^{1/4}} \quad \text{for } (\tau, \xi) \in \operatorname{supp} \widehat{F_2},
	\]
	we have
	\[
	\theta_T w_2 \lesssim T^{\epsilon/4} \Lambda_+^{-1} \Lambda_-^{\epsilon-1} F,
	\]
	whence
	\begin{equation} \label{theta w2}
		\|\theta_T w_2\|_{X^{x,b}} \lesssim T^{\epsilon/4} \| F \|_{X^{s-1, \theta+\epsilon-1}}.
	\end{equation}
	Adding \eqref{theta w0}, \eqref{theta w1} and \eqref{theta w2}, we have proved \eqref{energy11}. Due to the representation of solutions for linear waves, the function \( w \) defined in \eqref{1defw} satisfies the Cauchy problem \eqref{1linearw} for \( (t, x) \in [0, T] \times \mathbb{R}^d \). By standard energy estimates, \( w \) is the unique solution of \eqref{1linearw}. Therefore, the proof of Proposition~\ref{energy1dim} is complete.
	
	\section*{Acknowledgments}
	The author Huali Zhang is supported by Natural Science Foundation of Hunan Province, China (Grant No. 2025JJ40003) and the Fundamental Research Funds for the Central Universities (Grant No. 531118010867).
	
	\section*{Data Availability}
	The authors also confirm that the data supporting the findings of this study are available within the article.


{\small \begin{thebibliography}{plain}
		\bibitem{AFS1} P. D'Ancona, D. Foschi, and S. Selberg, {\it Product estimates for wave--Sobolev spaces in $2+1$ and $1+1$ dimensions}, Contemp. Math. {\bf 526} (2010), 125--150.
			
		\bibitem{AnKapi} L. Andersson and L. Kapitanski, {\it Cauchy problem for incompressible neo-Hookean materials}, Arch. Ration. Mech. Anal. {\bf 247} (2023), no.~2, Paper No. 21, 76 pp.
			
		\bibitem{ATY} K. Arthur, D.~H. Tchrakian, and Y. Yang, {\it Topological and nontopological self-dual Chern-Simons solitons in a gauged ${\rm O}(3)$ $\sigma$ model}, Phys. Rev. D {\bf 54} (1996), no.~8, 5245--5258.
			
		\bibitem{BBS} L. Berg\'e, A. de Bouard, and J.-C. Saut, {\it Blowing up time-dependent solutions of the planar, Chern-Simons gauged nonlinear Schrödinger equation}, Nonlinearity {\bf 8} (1995), no.~2, 235--253.
			
		\bibitem{CC} D. Chae and K. Choe, {\it Global existence in the Cauchy problem of the relativistic Chern-Simons-Higgs theory}, Nonlinearity {\bf 15} (2002), no.~3, 747--758.
			
		\bibitem{CCS} J.-L. Chern, Z.-Y. Chen, and H.-Y. Shen, {\it Classification of solutions for self-dual Chern-Simons $CP(1)$ model}, J. Math. Phys. {\bf 62} (2021), no.~3, Paper No. 031510, 25 pp.
			
		\bibitem{CH} K. Choe and J. Han, {\it Existence and properties of radial solutions in the self-dual Chern-Simons $O(3)$ sigma model}, J. Math. Phys. {\bf 52} (2011), no.~8, 082301, 20 pp.
			
		\bibitem{CHLL1} K. Choe et al., {\it Uniqueness and solution structure of nonlinear equations arising from the Chern-Simons gauged $O(3)$ sigma models}, J. Differential Equations {\bf 255} (2013), no.~8, 2136--2166.
			
		\bibitem{CHLL2} K. Choe et al., {\it Bubbling solutions for the Chern-Simons gauged $O(3)$ sigma model on a torus}, Calc. Var. Partial Differential Equations {\bf 54} (2015), no.~2, 1275--1329.
			
		\bibitem{CN} K. Choe and H. Nam, {\it Existence and uniqueness of topological multivortex solutions of the self-dual Chern–Simons CP(1) model}, Nonlinear Anal. {\bf 66} (2007), no.~12, 2794--2813.
			
		\bibitem{CS} S. Chern and J. Simons, {\it Some cohomology classes in principal fiber bundles and their application to Riemannian geometry}, Proc. Nat. Acad. Sci. U.S.A. {\bf 68} (1971), 791--794.
			
		\bibitem{DFW} W. Dai, D. Fang, and C. Wang, {\it Long-time existence for semilinear wave equations with the inverse-square potential}, J. Differential Equations {\bf 309} (2022), 98--141.
			
		\bibitem{FW} D. Fang and C. Wang, {\it Local well-posedness and ill-posedness on the equation of type $\square u=u^k(\partial u)^{\alpha}$}, Chinese Ann. Math. Ser. B {\bf 26} (2005), no.~3, 361--378.
			
		\bibitem{FK} D. Foschi and S. Klainerman, {\it Bilinear space-time estimates for homogeneous wave equations}, Ann. Sci. \'Ecole Norm. Sup. (4) {\bf 33} (2000), no.~2, 211--274.
			
		\bibitem{GG} P. K. Ghosh and S. K. Ghosh, {\it Topological and Nontopological Solitons in a Gauged $O(3)$ Sigma Model with Chern-Simons term}, Phys. Lett. B {\bf 366} (1996), no.~1-4, 199--204.
			
		\bibitem{GN} V. Grigoryan and A.~R. Nahmod, {\it Almost critical well-posedness for nonlinear wave equations with $Q_{\mu\nu}$ null forms in 2D}, Math. Res. Lett. {\bf 21} (2014), no.~2, 313--332.
			
		\bibitem{HJ} H. Huh and G. Jin, {\it Local and global solutions of Chern-Simons gauged $O(3)$ sigma equations in one space dimension}, J. Math. Phys. {\bf 57} (2016), no.~8, 081511, 20 pp.
			
		\bibitem{Huh1} H. Huh, {\it Low regularity solutions of the Chern-Simons-Higgs equations}, Nonlinearity {\bf 18} (2005), no.~6, 2581--2589.
			
		\bibitem{Huh2} H. Huh, {\it Local and global solutions of the Chern-Simons-Higgs system}, J. Funct. Anal. {\bf 242} (2007), no.~2, 526--549.
			
		\bibitem{Huh4} H. Huh, {\it Cauchy problem for the fermion field equation coupled with the Chern–Simons gauge}, Lett. Math. Phys. {\bf 79} (2007), no.~1, 75--94.
			
		\bibitem{Huh5} H. Huh, {\it Energy solution to the Chern-Simons-Schr\"odinger equations}, Abstr. Appl. Anal. {\bf 2013} (2013), no.~1, p. 590653.
			
		\bibitem{HO} H. Huh and S.-J. Oh, {\it Low regularity solutions to the Chern-Simons-Dirac and the Chern-Simons-Higgs equations in the Lorenz gauge}, Comm. Partial Differential Equations {\bf 41} (2016), no.~3, 375--397.
			
		\bibitem{JZ} G. Jin and H. Zhang, {\it Local well-posedness for Chern-Simons gauged $O(3)$ sigma equations under the Lorenz gauge}, arXiv preprint arXiv:2503.13871 (2025).
			
		\bibitem{KLL} K. Kimm, K. Lee, and T. Lee, {\it Anyonic Bogomol’nyi solitons in a gauged $O(3)$ $\sigma$ model}, Phys. Rev. D {\bf 53} (1996), no.~8, 4436--4440.
			
		\bibitem{KM1} S. Klainerman and M. Machedon, {\it Space‐time estimates for null forms and the local existence theorem}, Comm. Pure Appl. Math. {\bf 46} (1993), no.~9, 1221--1268.
			
		\bibitem{KM2} S. Klainerman and M. Machedon, {\it Smoothing estimates for null forms and applications}, Duke Math. J. {\bf 81} (1996), no.~1, 99-133.
			
		\bibitem{KM3} S. Klainerman, {\it Long time behaviour of solutions to nonlinear wave equations}, in {\it Proceedings of the International Congress of Mathematicians, Vol.\ 1, 2 (Warsaw, 1983)}, 1209--1215, PWN, Warsaw.
		
		\bibitem{KS} S. Klainerman and S. Selberg, {\it Bilinear estimates and applications to nonlinear wave equations}, Commun. Contemp. Math. {\bf 4} (2002), no.~2, 223--295.
			
		\bibitem{KT} M. Keel and T. Tao, {\it Local and global well-posedness of wave maps on $\mathbb{R}^{1+1}$ for rough data}, Internat. Math. Res. Notices (1998), no.~21, 1117--1156.
		
		\bibitem{L} R.~A. Leese, {\it Low-energy scattering of solitons in the ${\bf C}{\rm P}^1$ model}, Nuclear Phys. B {\bf 344} (1990), no.~1, 33--72.
			
		\bibitem{Lim} Z.~M. Lim, {\it Large data well-posedness in the energy space of the Chern-Simons-Schr\"odinger system}, J. Differential Equations {\bf 264} (2018), no.~4, 2553--2597.
			
		\bibitem{LST} B. Liu, P. Smith, and D. Tataru, {\it Local wellposedness of Chern–Simons–Schr\"{o}dinger}, Int. Math. Res. Not. IMRN {\bf 2014}, no.~23, 6341--6398.
		
		\bibitem{RS} I. Rodnianski and J. Sterbenz, {\it On the formation of singularities in the critical $O(3)$ $\sigma$-model}, Ann. of Math. (2) {\bf 172} (2010), no.~1, 187--242.
			
		\bibitem{Sch} B. J. Schroers, {\it Bogomol'nyi solitons in a gauged $O(3)$ sigma model}, Phys. Lett. B {\bf 356} (1995), no.~2--3, 291--296.
			
		\bibitem{Selberg} S. Selberg, {\it Multilinear space-time estimates and applications to local existence theory for nonlinear wave equations}, Ph.D. thesis, Princeton University, ProQuest LLC, Ann Arbor, MI, 1999.
			
		\bibitem{ST} S. Selberg and A. Tesfahun, {\it Global well-posedness of the Chern-Simons-Higgs equations with finite energy}, Discrete Contin. Dyn. Syst. {\bf 33} (2012), 2531--2546.
			
		\bibitem{STa} J. Sterbenz and D. Tataru, {\it Energy dispersed large data wave maps in $2+1$ dimensions}, Comm. Math. Phys. {\bf 298} (2010), no.~1, 139--230.
			
		\bibitem{Ta1} D. Tataru, {\it The $X^s_\theta$ spaces and unique continuation for solutions to the semilinear wave equation}, Comm. Partial Differential Equations {\bf 21} (1996), no.~5--6, 841--887.
			
		\bibitem{Ta2} D. Tataru, {\it Strichartz estimates in the hyperbolic space and global existence for the semilinear wave equation}, Trans. Amer. Math. Soc. {\bf 353} (2001), no.~2, 795--807.
			
		\bibitem{Ta3} D. Tataru, {\it Local and global results for wave maps. I}, Comm. Partial Differential Equations {\bf 23} (1998), no.~9--10, 1781--1793.
			
		\bibitem{Ta4} D. Tataru, {\it The wave maps equation}, Bull. Amer. Math. Soc. (N.S.) {\bf 41} (2004), no.~2, 185--204.
			
		\bibitem{WZ} S. Wang and Y. Zhou, {\it Physical space approach to wave equation bilinear estimates revisit}, Ann. PDE {\bf 10} (2024), no.~2, Paper No. 11, 14 pp.
			
		\bibitem{Yang} Y. Yang, {\it The existence of solitons in gauged sigma models with broken symmetry: some remarks}, Lett. Math. Phys. {\bf 40} (1997), no.~2, 177--189.
			
		\bibitem{Zhang} H. Zhang, {\it Local well-posedness for incompressible neo-Hookean elastic equations in almost critical Sobolev spaces}, Calc. Var. Partial Differential Equations {\bf 63} (2024), no.~3, Paper No. 66, 20 pp.
			
			
		\bibitem{Zhou2} Y. Zhou, {\it Local existence with minimal regularity for nonlinear wave equations}, Amer. J. Math. {\bf 119} (1997), no.~3, 671--703.
			
			
			
		\bibitem{Zhou5} Y. Zhou, {\it (1+2)-dimensional radially symmetric wave maps revisit}, Chinese Ann. Math. Ser. B {\bf 43} (2022), no.~5, 785--796.
	\end{thebibliography}}
	\end{document}